\documentclass[CJK,12pt,reqno]{amsart}
\usepackage{mathrsfs}
\usepackage{amssymb}
\usepackage{cases}
\usepackage{amsmath}
\usepackage{stmaryrd}
\usepackage{txfonts}
\usepackage{indentfirst,latexsym,bm}
\usepackage{graphicx}
\usepackage{psfrag}
\usepackage{amsmath,amssymb}
\usepackage{amsthm}
\usepackage[]{harpoon}
\usepackage[active]{srcltx} 
\usepackage[]{amsfonts}
\usepackage[]{amssymb}
\usepackage[]{harpoon}
\usepackage{epsfig}
\usepackage{multirow}
\usepackage{booktabs}
\usepackage{caption}
\usepackage{epstopdf}
\usepackage{subfigure}
\usepackage{color}%
\begin{document}
\newtheorem{thm}{Theorem}[section]
\newtheorem{cor}[thm]{Corollary}
\newtheorem{lem}[thm]{Lemma}
\newtheorem{prop}[thm]{Proposition}
\theoremstyle{definition}
\newtheorem{defn}[thm]{Definition}
\theoremstyle{Assertion}
\newtheorem{asser}[thm]{Assertion}
\theoremstyle{remark}
\newtheorem{rem}[thm]{Remark}
\numberwithin{equation}{section}
\newcommand{\norm}[1]{\left\Vert#1\right\Vert}
\newcommand{\abs}[1]{\left\vert#1\right\vert}
\newcommand{\set}[1]{\left\{#1\right\}}
\newcommand{\Real}{\mathbb R}
\newcommand{\eps}{\varepsilon}
\newcommand{\To}{\longrightarrow}
\newcommand{\BX}{\mathbf{B}(X)}
\newcommand{\A}{\mathcal{A}}
\newcommand{\ts}{\textstyle}
\newcommand{\tg}{\mbox{\rm tg}}
\newcommand{\ctg}{\mbox{\rm ctg}}
\newcommand{\atg}{\tg^{-1}}
\newcommand{\actg}{\ctg^{-1}}
\newcommand{\asin}{\sin^{-1}}
\newcommand{\acos}{\cos^{-1}}
\newcommand{\dps}{\displaystyle}
\newcommand{\fnz}{\footnotesize}
\newcommand{\D}{\displaystyle}
\newcommand{\DF}[2]{\frac{\D#1}{\D#2}}
\newcommand{\scp}{\scriptstyle}

\title[GRP for modeling blood flows]
{A direct Eulerian GRP scheme for a blood flow model in arteries$^*$}%
\author[Sheng, Zhang and Zheng]{Wancheng Sheng$^{1}$ \qquad Qinglong Zhang$^{1}$\qquad Yuxi Zheng$^{2}$}
\thanks{$^*$Supported by NSFC 11771274.\\
\indent Email: mathwcsheng@t.shu.edu.cn (Wancheng Sheng), zhangqinglong@nbu.edu.cn (Qinglong Zhang), zheng@psu.edu(Yuxi Zheng)
}
\dedicatory{$^{1}$Department of Mathematics, Shanghai University,
Shanghai, 200444, P.R.China\\
\vskip 5pt
$^{2}$Department of Mathematics, The Pennsylvania State University, University Park, PA 16802, United States
}

\begin{abstract}
In this paper, we propose a direct Eulerian generalized Riemann problem (GRP) scheme for a blood flow model in arteries. It is an extension of the Eulerian GRP scheme, which is developed by Ben-Artzi, et. al. in J. Comput. Phys., 218(2006). By using the Riemann invariants, we diagonalize the blood flow system into a weakly coupled system, which is used to resolve rarefaction wave. We also use Rankine-Hugoniot condition to resolve the local GRP formulation. We pay special attention to the acoustic case as well as the sonic case. The extension to the two dimensional case is carefully obtained by using the dimensional splitting technique. We test that the derived GRP scheme is second order accuracy.

\vskip 6pt

\noindent%

\vskip 6pt
\noindent%
\end{abstract}

\maketitle

\section{Introduction}
A simple set of equations for the blood flow in arteries are given by \cite{Toro2}
\begin{equation}\label{1.1}
\left\{\begin{array}{l}
A_t+(A u)_{x}=0,\\
\displaystyle (A u)_t+(Au^2)_x+\frac{A}{\rho}p_x=0,
\end{array}\right.
\end{equation}
where $A(x,t)$ is the cross section area of the vessel, $\rho, p, u$ represent the density, pressure and the averaged velocity of the blood, respectively. Here we treat $\rho$ as a constant. The pressure $p$ is given by
\begin{equation}\label{1.2}
p=k(x)\big(\alpha^m-\alpha^n\big), \quad \alpha=\frac{A}{A_e}.
\end{equation}
$k(x)$ is the stiffness coefficient of the vessel, which represents the elastic properties of the vessel. $A_e$ is the cross section area at equilibrium state, which is assumed to be constant here. $m\geq 0$ and $n\leq 0$ are two constants. For the blood flows in arteries, we take $m=1/2, n=0$, see \cite{ToroSiviglia}.

In recent years, many numerical schemes have been developed to deal with the blood flow model. For example, a well-balanced high order scheme has been constructed for flow in blood vessels with varying mechanical properties in \cite{Muller2}.  Both a discontinuous Galerkin and a Taylor-Galerkin formulations have been proposed in \cite{Sherwin}. In \cite{Wang}, the authors designed a high order finite difference weighted non-oscillatory (WENO) scheme for the blood flow model. For other examples, we refer to \cite{Contarino,Muller1,Muller3}.

The GRP scheme was originally designed for compressible fluid flows \cite{BenF1}. It was extensively applied to many problems, including gas dynamics \cite{BenF2,Han}, reaction flows \cite{BenF}, relativistic hydrodynamics \cite{Yang}, and also used for designing high order numerical schemes \cite{Li1}. The GRP scheme has two versions: the Lagrangian and the Eulerian. The passage from the Lagrangian version to the Eulerian is sometimes quite delicate, especially for sonic cases and multidimensional applications. In order to deal with that, a direct Eulerian GRP scheme was recently developed for the shallow water equations \cite{Li} as well as the Euler equations \cite{BenL1}.  The main ingredient of GRP is the use of Riemann invariants to decompose the system into a diagonalized case so that the rarefaction waves could be analytically resolved in a straightforward way. However, most quasilinear systems do not have such a set of Riemann invariants. In \cite{BenL2}, the authors have introduced a concept: weakly coupled systems. Such systems have ``partial set" of Riemann invariants which enables a ``diagonalized" treatment of GRP.  In this paper, we consider the generalized Riemann problem for the blood flow model in arteries. We know that in blood flow models, the slope of characteristic curves are not only dependent on the state variables $A$ and $u$, but also on the spatial coordinate $x$, which is different from  the Eulerian case.
This encourages us to find ways to resolve rarefaction waves. For shock waves, the Rankine-Hugoniot conditions do not apply to the system directly since it is not hyperbolic anymore. However, we will show that the function k(x) remains constant across a shock wave, which enables us to establish the shock relation between the two sides.

Considering smooth solutions, we write \eqref{1.1} as
\begin{equation}\label{1.3}
\frac{{\rm \partial} U}{{\rm \partial} t}+\frac{{\rm \partial} F(U)}{{\rm \partial} x}+L(U)\frac{{\rm \partial} G(x)}{{\rm \partial} x}=0,
\end{equation}
where $U=\left(A,Au\right)^{T}$, and
\begin{equation}\label{1.4}
\begin{array}{ll}
F(U)=\left( \begin{array}{cc}
Au \\
 \displaystyle Au^2+\frac{mk}{\rho(m+1)}\cdot\frac{A^{m+1}}{A_e^{m}}
\end{array} \right),
\
L(U)=\left(  \begin{array}{cc}
0 & 0 \\
0& \displaystyle \frac{A^{m+1}}{\rho(m+1)A_e^{m}}-\frac{A}{\rho}
\end{array} \right),
\
G(x)=\left(  \begin{array}{c}
0  \\
k(x)
\end{array} \right).
\end{array}
\end{equation}
To construct the numerical flux, we need to resolve the generalized Riemann problem at each cell interface, given the piecewise smooth initial data. Specially, we denote the computational cell $j$ as $C_j=[x_{j-1/2},x_{j+1/2}]$, $\Delta x=x_{j+1/2}-x_{j-1/2}$, and $\{t_n\}_{n=0}^{\infty}$ as the sequence of the discretized time levels, $\Delta t=t_{n+1}-t_{n}$. We assume that the data at time $t=t_n$ is piecewise linear with a slope $\sigma_j^{n}$, which is
\begin{equation}\label{1.5}
U(x,t_n)=U_j^{n}+\sigma_j^{n}(x-x_j), \quad x\in (x_{j-1/2}, x_{j+1/2})
\end{equation}
on the cell $C_j$ . Then a "quasi-conservative" scheme is given by
\begin{equation}\label{1.6}
\displaystyle U_j^{n+1}=U_j^{n}-\frac{\Delta t}{\Delta x}\left(F_{j+1/2}^{n+1/2}-F_{j-1/2}^{n+1/2}\right)-\frac{\Delta t}{2\Delta x}\left(L_{j+1/2}^{n+1/2}+L_{j-1/2}^{n+1/2}\right)\left(G_{j+1/2}^{n+1/2}-G_{j-1/2}^{n+1/2}\right),
\end{equation}
where
$$
\displaystyle F_{j\pm1/2}^{n+1/2}=F\left(U_{j\pm1/2}^{n+1/2}\right), \quad L_{j\pm1/2}^{n+1/2}=L\left(U_{j\pm1/2}^{n+1/2}\right),\quad G_{j\pm1/2}^{n+1/2}=G\left(x_{j\pm1/2}^{n+1/2}\right),
$$
$U_j^n$ is the average of $U(x,t_n)$ over the cell $C_j$ and $U_{j\pm1/2}^{n+1/2}$ is the average of $U(x_{j\pm1/2},t)$ over the time interval $[t_n,t_{n+1}]$. The source term is discretized by the trapezoidal rule in space and the mid-point rule in time \cite{BenF2,Jin}. We are left with how to obtain the mid-point value $U_{j+1/2}^{n+1/2}$. We approximate this by the Taylor expansion ignoring the higher order terms,
\begin{equation}\label{1.7}
U_{j+1/2}^{n+1/2}\cong U_{j+1/2}^{n}+\frac{\Delta t}{2}\left(\frac{\partial U}{\partial t}\right)_{j+1/2}^{n},
\end{equation}
where
\begin{equation}\label{1.8}
\displaystyle U_{j+1/2}^{n}=\lim_{t\to t_n+0}U(x_{j+1/2},t),\quad \left(\frac{\partial U}{\partial t}\right)_{j+1/2}^{n}=\lim_{t\to t_n+0}\frac{\partial U}{\partial t}(x_{j+1/2},t).
\end{equation}
The value $U_{j+1/2}^{n}$ is obtained by solving the associated Riemann problem for the homogeneous hyperbolic conservative equations. Thus the main ingredient of our GRP solution is left with the calculation of $\left(\frac{\partial U}{\partial t}\right)_{j+1/2}^{n}$.
For the smooth solution $U$ near the grid point $(x_{j+1/2},t_n)$, we use \eqref{1.3} to get
\begin{equation}\label{1.9}
\left( \frac{{\rm \partial} U}{{\rm \partial} t}\right)_{j+1/2}^{n}=- F'(U_{j+1/2}^{n})\left(\frac{\partial U}{\partial x}\right)_{j+1/2}^{n}-L(U_{j+1/2}^{n})\left(\frac{{\rm \partial} G}{{\rm \partial} x}\right)_{j+1/2}^{n}.
\end{equation}
But this is not valid for the generalized Riemann problem including singularity at $(x_{j+1/2},t_n)$. We then want to look for a substitute of \eqref{1.9}. Since the characteristic fields are not only dependent on $A$ and $u$, but also dependent explicitly on $x$, there is no Riemann invariants to diagonalize the system. However, the system is endowed with a coordinate system of Riemann invariants (\cite{Li2}), such that \eqref{1.1} can be reduced to the quasi-diagonal form. Thus we can define the characteristic coordinate to resolve the rarefaction waves in the generalized Riemann problem and obtain the resultant GRP scheme.
The construction of the GRP scheme  for \eqref{1.3} is divided into two steps: (i) We use the exact (or approximate) Riemann solver in each grid point $(x_{j+1/2},t_n)$, see \cite{Toro1,Toro3} . (ii) We obtain the limiting value $\left(\frac{\partial U}{\partial t}\right)_{j+1/2}^{n}$ by solving a linear algebraic system.

This paper is organized as follows. In Section 2, we analyze some basic properties of elementary waves for blood flow model \eqref{1.1} and some basic setup for the GRP.  The resolution of rarefaction waves and shock waves are given in details in Sections 3 and 4 respectively. We make our conclusions on the GRP scheme and the acoustic approximation in Section 5.  In Section 6, we extend the GRP to the two-dimensional blood flow model by a dimensional splitting technique. We outline the implementation in Section 7 and provide some numerical tests including 1-D and 2-D cases in the last section.

\section{Preliminaries and notations}
\subsection{Characteristic analysis and elementary waves}
In this section, we present some preliminaries of the generalized Riemann problem. System \eqref{1.3} has two eigenvalues
\begin{equation}\label{2.2}
\lambda_1=u-c,\quad \lambda_2=u+c,
\end{equation}
where $\displaystyle c=\sqrt{\frac{mk}{\rho}\left(\frac{A}{A_e}\right)^{m}}$. We see that $\lambda_1$ and $\lambda_2$ depend explicitly on the spatial coordinate $x$, which is different from the Eulerian case. To resolve rarefaction waves, we introduce two variables $\phi$ and $\psi$
\begin{equation}\label{2.3}
\phi=u-\int^{A}\frac{c(\omega)}{\omega}{\rm d}\omega,\quad \psi=u+\int^{A}\frac{c(\omega)}{\omega}{\rm d}\omega.
\end{equation}
For the blood flow model in arteries, we take $m=1/2$.  \eqref{2.3} is equivalent to
\begin{equation}\label{2.4}
\displaystyle \phi=u-\frac{2}{m}c,\quad \psi=u+\frac{2}{m}c.
\end{equation}
We should note here that $\phi$ and $\psi$ are not only dependent on $A$ and $u$, but also dependent on the spatial coordinate $x$. Since there is no full coordinate system of Riemann invariant, \eqref{1.3} cannot be reduced to a diagonal form. However, it falls into a type of weakly coupled system \cite{BenL1}. Indeed, from \eqref{2.4}, we have by using \eqref{1.1} that
\begin{equation}\label{2.5'}
\begin{array}{ll}
\displaystyle \frac{\partial \phi}{\partial t}+(u-c) \frac{\partial \phi}{\partial x}=u_t-\frac{c}{A}A_t+(u-c)\left(u_x-\frac{c}{A}A_x-\frac{c}{mk}k_x\right)=\frac{k_x}{\rho}-\frac{uck_x}{mk}.
\end{array}
\end{equation}
Similarly, we have
\begin{equation}\label{2.6'}
\begin{array}{ll}
\displaystyle \frac{\partial \psi}{\partial t}+(u+c) \frac{\partial \psi}{\partial x}=u_t+\frac{c}{A}A_t+(u+c)\left(u_x+\frac{c}{A}A_x+\frac{c}{mk}k_x\right)=\frac{k_x}{\rho}+\frac{uck_x}{mk}.
\end{array}
\end{equation}
Thus we take $\phi$, $\psi$ as (quasi)-Riemann invariants to rewrite \eqref{1.3} as
\begin{equation}\label{2.5}
\left\{
\begin{array}{ll}
\displaystyle \frac{\partial \phi}{\partial t}+(u-c) \frac{\partial \phi}{\partial x}=B_1,\\[12pt]
\displaystyle \frac{\partial \psi}{\partial t}+(u+c) \frac{\partial \psi}{\partial x}=B_2,
\end{array}
\right.
\end{equation}
where $\displaystyle B_1=\frac{k_x}{mk}\cdot\left(\frac{mk}{\rho}-uc\right)$, $\displaystyle B_2=\frac{k_x}{mk}\cdot\left(\frac{mk}{\rho}+uc\right)$.
We use the weakly coupled form \eqref{2.5} to resolve the generalized Riemann problem for \eqref{1.3} subject to the initial data
\begin{equation}\label{2.6}
U(x,0)=\left\{
\begin{array}{ll}
U_L+U_L^{'}x,\quad x<0,\\[9pt]
U_R+U_R^{'}x,\quad x>0,
\end{array}
\right.
\end{equation}
where $U_L, U_R, U_L^{'},U_R^{'}$ are constant vectors.
The initial structure of the solution is determined by the associated Riemann problem
\begin{equation}\label{2.7}
\begin{array}{ll}
\displaystyle \frac{{\rm \partial} U}{{\rm \partial} t}+\frac{{\rm \partial} F(U)}{{\rm \partial} x}=0,\end{array}
\end{equation}

\begin{figure}[htbp]
\subfigure{
\begin{minipage}[t]{0.48\textwidth}
\centering
\includegraphics[width=\textwidth]{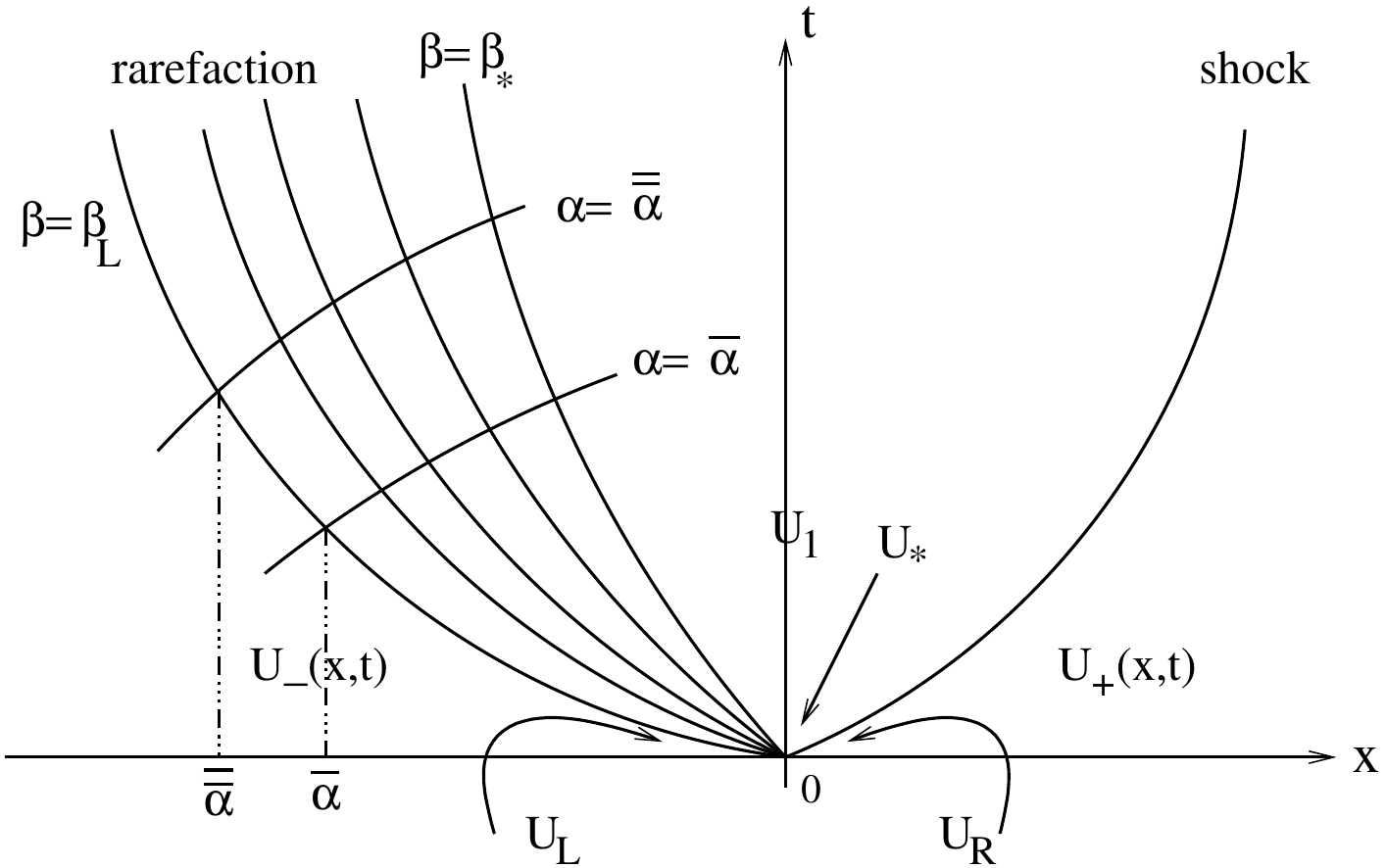}
\end{minipage}
}
\caption*{(a) Wave patterns for GRP. The initial data (\ref{2.6}).}
\subfigure{
\begin{minipage}[t]{0.48\textwidth}
\centering
\includegraphics[width=\textwidth]{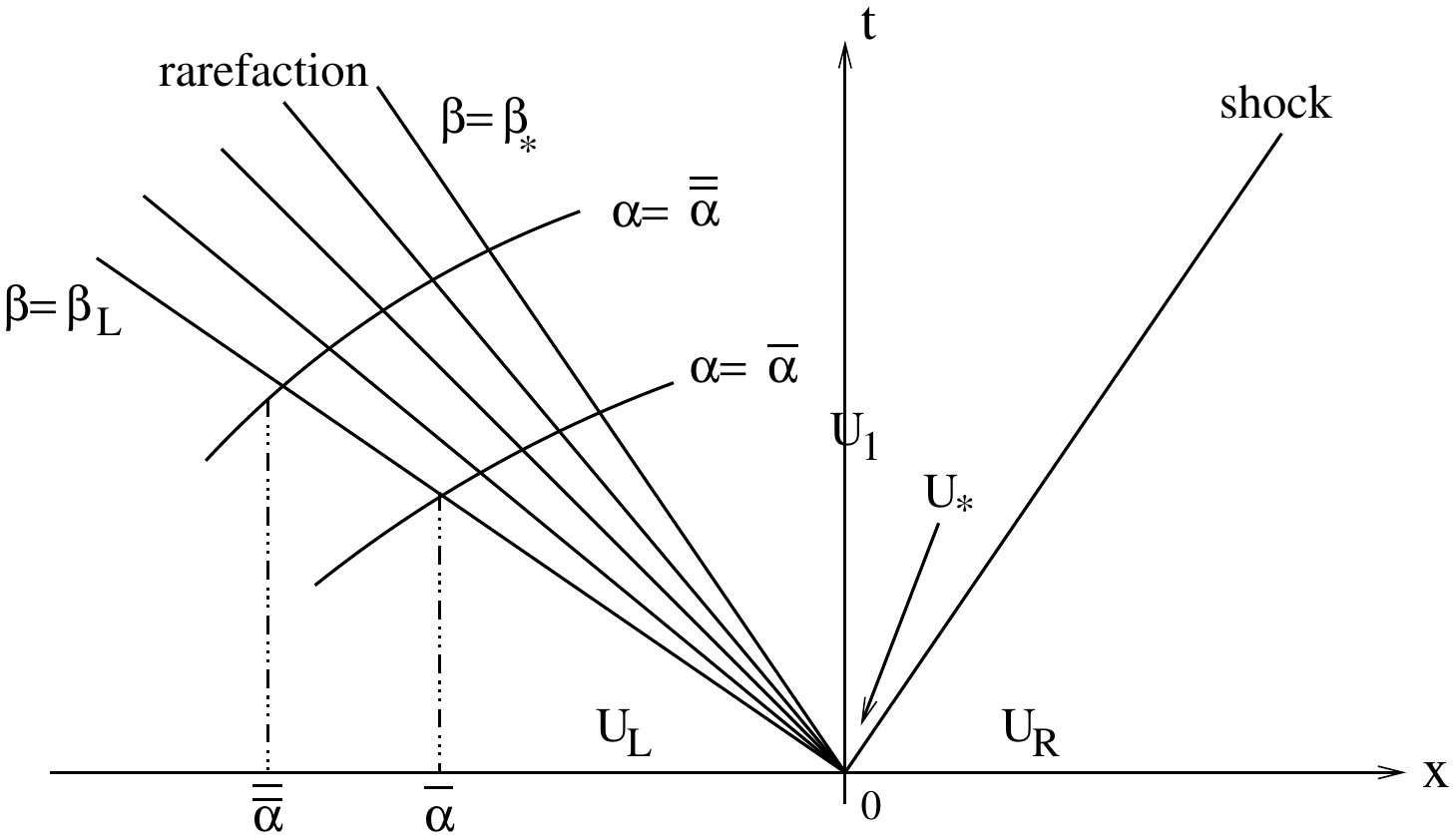}
\end{minipage}
}
\caption*{(b) Wave patterns for the associated Riemann problem.}
\caption*{{Fig.2.1. Typical wave configuration.}}
\end{figure}
\noindent%
with
\begin{equation}\label{2.8}
U(x,0)=\left\{
\begin{array}{ll}
U_L, \quad x<0,\\
U_R, \quad x>0.
\end{array}
\right.
\end{equation}
Denote $R^{A}(x/t;U_L,U_R)$  as the Riemann solution of \eqref{2.7} with \eqref{2.8}, we have the following proposition.
\begin{prop}
Let $U(x,t)$ be the solution to the generalized Riemann problem $\eqref{1.3}$ with \eqref{2.6}. Then for the fixed direction $\lambda=x/t$,
\begin{equation}\label{2.9}
\lim_{t\to 0+}{U(\lambda t, t)}=R^{A}(\lambda; U_L,U_R).
\end{equation}
This implies that the wave configuration for the generalized Riemann problem \eqref{1.3} with \eqref{2.6} is the same with that for the associated Riemann problem \eqref{2.7} with \eqref{2.8} near the origin $(x,t)=(0,0)$.
\end{prop}
We illustrate the Proposition in Fig. 2.1. The solution of \eqref{1.3} (a) is curved and the corresponding solution of \eqref{2.7} (b) is self-similar. See \cite{BenL1} for more details.

Now we use the weakly coupled form to resolve the generalized Riemann problem \eqref{1.3} subject to initial data \eqref{2.6}.
Consider the wave configuration in Fig.2.1, a rarefaction wave moves to the left and a shock wave moves to the right. The intermediate state is denoted by $U_1$.
Denote $U_{*}$ as the limiting state at $x=0$ as $t\to 0$.

\section{The resolution of rarefaction waves}
Consider the rarefaction wave associated with $u-c$ from the left, see Fig. 2.1. Since the system is endowed with a coordinate system of Riemann invariants $\phi$ and $\psi$ (\cite{BenL1}), \eqref{2.5} can be written as
\begin{equation}\label{3.1'}
\left\{
\begin{array}{ll}
\displaystyle \frac{\partial \phi}{\partial t}+\lambda_1(\phi,\psi) \frac{\partial \phi}{\partial x}=B_1(x,\phi,\psi),\\[12pt]
\displaystyle \frac{\partial \psi}{\partial t}+\lambda_2(\phi,\psi) \frac{\partial \psi}{\partial x}=B_2(x,\phi,\psi).
\end{array}
\right.
\end{equation}
The characteristic coordinates play an important role in the resolution of rarefaction waves. Denote the characteristic curves throughout the rarefaction wave as $\beta=\beta(x,t)$ and $\alpha=\alpha(x,t), \beta \in [\beta_L,\beta_{*}], -\infty\leq \alpha<0.$ $\beta_L=u_L-c_L$, $\beta_{*}=u_{*}-c_{*}$. Let $C^{\lambda_1}: \beta(x,t)=const.$ be the integral curve of the differential equation
\begin{equation}\label{2.10}
\displaystyle \frac{{\rm d}x}{{\rm d}t}=\lambda_1(\phi, \psi),
\end{equation}
and
$C^{\lambda_2}: \alpha(x,t)=const.$ be the integral curve of the differential equation
\begin{equation}\label{2.10'}
\displaystyle \frac{{\rm d}x}{{\rm d}t}=\lambda_2(\phi, \psi).
\end{equation}
Thus the coordinates $(x,t)$ can be expressed as functions of $\alpha$ and $\beta$. This transformation is denoted by
\begin{equation}\label{2.11}
\displaystyle x=x(\alpha,\beta),\quad t=t(\alpha,\beta).
\end{equation}
In terms of the characteristic coordinates ($\alpha$, $\beta$), system \eqref{3.1'} can be rewritten in the following form
\begin{equation}\label{2.12}
\left\{
\begin{array}{ll}
\displaystyle \frac{\partial x}{\partial \alpha}=\lambda_1 \frac{\partial t}{\partial \alpha},\quad \frac{\partial \phi}{\partial \alpha}=\frac{\partial t}{\partial \alpha}\cdot B_1(x(\alpha,\beta),\phi,\psi),\\[12pt]
\displaystyle \frac{\partial x}{\partial \beta}=\lambda_2 \frac{\partial t}{\partial \beta}, \quad \frac{\partial \psi}{\partial \beta}=\frac{\partial t}{\partial \beta}\cdot B_2(x(\alpha,\beta),\phi,\psi).
\end{array}
\right.
\end{equation}
This implies
\begin{equation}\label{2.13}
\displaystyle \frac{\partial^{2} t}{\partial \alpha \partial \beta}=\frac{1}{\lambda_2-\lambda_1}\cdot \left(\frac{\partial \lambda_1}{\partial \beta}\frac{\partial t}{\partial \alpha}-\frac{\partial \lambda_2}{\partial \alpha}\frac{\partial t}{\partial \beta}\right).
\end{equation}
Similarly, a local coordinate transformation is also introduced within the region of the left rarefaction wave in Fig. 2.1. We denote it by $x_{ass}=x(\alpha,\beta), t_{ass}=t(\alpha,\beta)$. Following the same derivation in \cite{BenL2}, we have that
\begin{equation}\label{2.14}
\displaystyle  t_{ass}(\alpha,\beta)=\frac{\alpha }{\left(\psi_L-\beta\right)^{\frac{m+2}{2m}}},\quad x_{ass}(\alpha,\beta)=\frac{\alpha \beta}{\left(\psi_L-\beta\right)^{\frac{m+2}{2m}}}.
\end{equation}
where $\psi_{L}=\psi(u_L,A_L)$. Particularly, at $\alpha=0$ we have
\begin{equation}\label{2.15}
\displaystyle  \frac{\partial \lambda_1}{\partial \beta}(0,\beta)=1, \quad \frac{\partial t}{\partial \alpha}(0,\beta)=\frac{\partial t_{ass}}{\partial \alpha}(0,\beta),\quad  \frac{\partial t}{\partial \beta}(0,\beta)=0, \quad \beta_L\leq \beta\leq \beta_{*}.
\end{equation}
Moreover, in the limit case $\alpha\to 0$, we have $u(0,\beta)-c(0,\beta)=\beta$ so that $c=\mu\left(\psi_L-\beta\right)$ and $u=(\mu-1)\left(\psi_L-\beta\right)+\psi_L$, where $\mu=\frac{m}{m+2}$. The linear relations of $\partial u/\partial t$ and $\partial A/\partial t$ are stated in the following lemma.
\begin{lem}
Assume that the rarefaction wave associated with $u-c$ moves to the left, see Fig. 2.1.
Use the characteristic coordinates $(\alpha,\beta)$ as above and the limiting  values $\frac{\partial u}{\partial t}(0,\beta)$ and $\frac{\partial p}{\partial t}(0,\beta), \beta_L\leq \beta\leq \beta_{*}$ satisfies the following relation
\begin{equation}\label{2.16}
a_L(0,\beta)\cdot \frac{\partial u}{\partial t}(0,\beta)+b_L(0,\beta)\cdot \frac{\partial A}{\partial t}(0,\beta)=d_L(\beta),
\end{equation}
where the coefficients $a_L, b_L$ and  $d_L$ are expressed as follows.
\begin{equation}\label{2.17}
\left(a_L(0,\beta),b_L(0,\beta)\right)=\left(1, \frac{c(0,\beta)}{A(0,\beta)}\right),
\end{equation}
and
\begin{equation}\label{2.18}
\displaystyle  d_L(\beta)=\frac{\beta+2c(0,\beta)}{2c(0,\beta)}\cdot (\psi_L-\beta)^{\frac{m+2}{2m}}\cdot \frac{\partial \psi}{\partial \alpha}(0,\beta)-\frac{\beta}{2c(0,\beta)}\cdot \left(\frac{k^{'}(0)}{\rho}+\frac{\left(\beta+c(0,\beta)\right)\cdot c(0,\beta)}{m}\cdot\frac{k^{'}(0)}{k(0)}\right),
\end{equation}
where  $\beta=u(0,\beta)-c(0,\beta)$, $\theta=c(0,\beta)/c_L$. $\displaystyle \frac{\partial \psi}{\partial \alpha}(0,\beta)$ is given explicitly by

\begin{equation}\label{2.19}
\begin{array}{ll}
\displaystyle \frac{\partial \psi}{\partial \alpha}(0,\beta)=\frac{\partial \psi}{\partial \alpha}(0,\beta_L)+\frac{k'(0)}{\rho}\left(\frac{m+2}{m}c_L\right)^{-\frac{m+2}{2m}}(\theta^{-\frac{m+2}{2m}}-1)
\\[9pt]
~~\displaystyle  -\frac{\psi_L}{m-2}\frac{k'(0)}{k(0)}\left(\frac{m+2}{m}c_L\right)^{\frac{m-2}{2m}}(\theta^{\frac{m-2}{2m}}-1)+
\frac{2}{(m+2)(3m-2)}\frac{k'(0)}{k(0)}\left(\frac{m+2}{m}c_L\right)^{\frac{3m-2}{2m}}(\theta^{\frac{3m-2}{2m}}-1)
\end{array}
\end{equation}
with
\begin{equation}\label{2.20}
\frac{\partial \psi}{\partial \alpha}(0,\beta_L)=\left(\frac{m+2}{m}c_L\right)^{-\frac{m+2}{2m}}\cdot \left(\frac{k^{'}(0)}{\rho}+\frac{u_Lc_L}{m}\cdot \frac{k^{'}(0)}{k(0)}-2c_L\psi_L^{'}\right).
\end{equation}
\end{lem}

\begin{proof}
The equation for $\psi$ in \eqref{2.4} yields
\begin{equation}\label{2.21}
\frac{\partial \psi}{\partial t}=\frac{\partial u}{\partial t}+ \frac{c}{A}\frac{\partial A}{\partial t},
\end{equation}
so we only need to compute $\frac{\partial \psi}{\partial t}$.
From \eqref{2.12} we have
\begin{equation}\label{2.22}
\left\{
\begin{array}{ll}
\displaystyle \frac{\partial \psi}{\partial \beta}=\frac{\partial t}{\partial \beta}\cdot\left(\frac{\partial \psi}{\partial t}+(u+c)\frac{\partial \psi}{\partial x}\right),\\[12pt]
\displaystyle  \frac{\partial \psi}{\partial \alpha}=\frac{\partial t}{\partial \alpha}\cdot\left(\frac{\partial \psi}{\partial t}+(u-c)\frac{\partial \psi}{\partial x}\right),
\end{array}
\right.
\end{equation}
together with \eqref{2.5} and \eqref{2.15}, we have
\begin{equation}\label{2.23}
\frac{\partial}{\partial \beta}\left(\frac{\partial \psi}{\partial \alpha}(0,\beta)\right)=\frac1{2c(0,\beta)}\cdot \frac{\partial t}{\partial \alpha}(0,\beta)\cdot B_2(0,\beta).
\end{equation}
The integration from $\beta_L$ to $\beta$ yields,
\begin{equation}\label{2.24}
\frac{\partial \psi}{\partial \alpha}(0,\beta)=\frac{\partial \psi}{\partial \alpha}(0,\beta_L)+\int_{\beta_L}^{\beta}\frac1{2c(0,\eta)}\cdot \frac{\partial t}{\partial \alpha}(0,\eta)\cdot B_2(0,\eta){\rm d}\eta.
\end{equation}
The initial data of $\partial \psi/\partial \alpha$ is given by
\begin{equation}\label{2.25}
\frac{\partial \psi}{\partial \alpha}(0,\beta_L)=\left(\frac{m+2}{m}c_L\right)^{-\frac{m+2}{2m}}\cdot \left(\frac{k^{'}(0)}{\rho}+\frac{u_Lc_L}{m}\cdot \frac{k^{'}(0)}{k(0)}-2c_L\psi_L^{'}\right).
\end{equation}
The integration term in \eqref{2.24} can be obtained directly by using \eqref{2.14} and \eqref{2.15}.

Once we obtain $\partial \psi/\partial \alpha(0,\beta)$ from \eqref{2.24}, we use the following relation from \eqref{2.5} and \eqref{2.22},
\begin{equation}\label{2.26}
\displaystyle  \frac{\partial \psi}{\partial \alpha}=\frac{\partial t}{\partial \alpha}\cdot\left(\frac{\partial \psi}{\partial t}+(u-c)\frac{\partial \psi}{\partial x}\right)=\frac{\partial t}{\partial \alpha}\cdot\left(B_2-2c\frac{\partial \psi}{\partial x}\right)
\end{equation}
to get the expression of $\partial \psi/\partial x$. Since in \eqref{2.22}, we note
\begin{equation}\label{2.27}
\displaystyle  \frac{\partial \psi}{\partial t}+(u-c) \cdot \frac{\partial \psi}{\partial x}=\left(\frac{\partial t}{\partial \alpha}\right)^{-1}\cdot\frac{\partial \psi}{\partial \alpha}.
\end{equation}
We insert \eqref{2.24} and \eqref{2.26} to \eqref{2.27}, and get
\begin{equation}\label{2.28}
\displaystyle  \frac{\partial \psi}{\partial t}(0,\beta)=-\frac{u-c}{2c}\cdot B_2(0,\beta)+\frac{u+c}{2c}\cdot \left(\frac{\partial t}{\partial \alpha}\right)^{-1}\cdot \frac{\partial \psi}{\partial \alpha}(0,\beta).
\end{equation}
Thus using $\displaystyle B_2(0,\beta)=\frac{k'(0)}{mk(0)}\cdot\left(\frac{mk(0)}{\rho}+u(0,\beta)\cdot c(0,\beta)\right)$, we obtain \eqref{2.18}.
\end{proof}

\section{The resolution of shock waves}
In this section, we resolve the shock wave at the origin. Let $x=x(t)$ be the shock trajectory associated with the second characteristic field, see Fig. 2.1. Denote the left- and right-hand states of the shock wave by $U$ and $\overline{U}$, i.e., $U(t) :=U(x(t-0), 0)$ and $\overline{U}(t) :=U(x(t+0), 0)$. Across the shock wave, the Rankine-Hugoniot relation can be written in the form:
\begin{equation}\label{4.1'}
\displaystyle \zeta(A, u, k, A_e, \bar{A}, \bar{u}, \bar{k}, \bar{A}_e)=0.
\end{equation}
Since $A_e$ is constant as assumed, $k(x)$ is a continuous function with respect to $x$. We have  $A_e=\bar{A}_e$ and $k=\bar{k}=k(x(t))$. Which indicates that $k(x(t))$  remains unchanged in a sufficient small domain of shock wave. Thus from \eqref{1.3}, we write the Rankine-Hugoniot relation as
\begin{equation}\label{4.2'}
\left\{
\begin{array}{ll}
\displaystyle \sigma[A]=[Au],\\
\displaystyle  \sigma[Au]=\left[Au^2+\frac{mk}{\rho(m+1)}\cdot\frac{A^{m+1}}{A_e^{m}}\right],
\end{array}
\right.
\end{equation}
where $``[f]"$ mean the jump of function $f$ across shock wave. By eliminating $\sigma$ in \eqref{4.2'}, we have
\begin{equation}\label{3.1}
u=\overline{u}+\Phi\left(A;\overline{A},\overline{u}\right),
\end{equation}
where
\begin{equation}\label{3.2}
\displaystyle \Phi\left(A;\overline{A},\overline{u},k\right)=\sqrt{\frac{mk}{\rho(m+1)A_e^{m}}\frac{(A-\overline{A})(A^{m+1}-\overline{A}^{m+1})}{A\overline{A}}}.
\end{equation}
The shock speed is given by
\begin{equation}\label{3.3}
\displaystyle \sigma=\frac{Au-\overline{A}\overline{u}}{A-\overline{A}}.
\end{equation}
We take the directional derivative along the shock trajectory $x=x(t)$ to get
\begin{equation}\label{3.4}
\displaystyle \left(\frac{\partial}{\partial t}+\sigma\frac{\partial}{\partial x}\right)\Gamma=0,
\end{equation}
where $\Gamma=u-\overline{u}-\Phi\left(A;\overline{A},\overline{u},k\right)$.
By taking the limit $t\to t_{0+}$, we have
\begin{equation}\label{3.5}
\frac{\partial A}{\partial t}\to \left(\frac{\partial A}{\partial t}\right)_{*},\quad \frac{\partial u}{\partial t}\to \left(\frac{\partial u}{\partial t}\right)_{*},\quad \frac{\partial \overline{U}}{\partial x}\to U_R^{'},
\end{equation}
and
\begin{equation}\label{3.6}
\left(A,u\right)\to \left(A_*,u_*\right),\quad \left(\overline{A},\overline{u}\right)\to \left(A_R,u_R\right).
\end{equation}
We resolve the shock wave and have the following lemma.
\begin{lem}
Assume that a shock wave associated with $u+c$ moves to the right. Then the limiting values $\left(\partial A/\partial t\right)_{*}$ and $\left(\partial u/\partial t\right)_{*}$  satisfy the linear relation
\begin{equation}\label{3.7}
\displaystyle a_R\left(\frac{\partial u}{\partial t}\right)_{*}+b_R\left(\frac{\partial A}{\partial t}\right)_{*}=d_R,
\end{equation}
where $a_R, b_R$ and $d_R$ are given explicitly by
\begin{equation}\label{3.8}
\left\{\begin{array}{llll}
\displaystyle a_R=1-\frac{\sigma_R }{u_{*}^2-c_{*}^2}\cdot \left(u_{*}+A_*\Phi_1\right), \quad b_R=\frac{\sigma_R}{u_{*}^2-c_{*}^2}\cdot \left(u_{*}\Phi_1+\frac{c_{*}^2}{A_*}\right)-\Phi_1,\\[10pt]
d_R=\displaystyle \left(\sigma_R-u_R\right)u_R^{'}-\frac{k^{'}(0)}{\rho}\cdot \left(\left(\frac{A_R}{A_e}\right)^{m}-1\right)-\frac{mk(0)A_R^{m-1}A_R^{'}}{\rho A_e^{m}}+\Phi_2 \cdot \left((\sigma_R-u_R)A_R^{'}-A_Ru_R^{'}\right)\\[12pt]
\qquad\displaystyle +\frac{\sigma_R k'(0)}{\rho(u_*^{2}-c_*^{2})}\cdot \left(\left(\frac{A_*}{A_e}\right)^{m}-1\right)\cdot \left(A_* \Phi_1+u_*\right)+\sigma_kk'(0)\Phi_k.
\end{array}\right.
\end{equation}
Here $\Phi_1$ and $\Phi_2$ are defined by
\begin{equation}\label{3.9}
\Phi_1=\frac{\partial \Phi}{\partial A}\left(A_*;A_R,u_R,k\right),\quad \Phi_2=\frac{\partial \Phi}{\partial \overline{A}}\left(A_*;A_R,u_R,k\right).
\end{equation}
\begin{proof}
Following \eqref{3.4}, we have
\begin{equation}\label{3.10}
\frac{\partial u}{\partial t}+\sigma_R\frac{\partial u}{\partial x}=\frac{\partial \overline{u}}{\partial t}+\sigma_R\frac{\partial \overline{u}}{\partial x}+\frac{\partial \Phi}{\partial A}\left(\frac{\partial A}{\partial t}+\sigma_R\frac{\partial A}{\partial x}\right)+\frac{\partial \Phi}{\partial \overline{A}}\left(\frac{\partial \overline{A}}{\partial t}+\sigma_R\frac{\partial \overline{A}}{\partial x}\right)+\frac{\partial \Phi}{\partial k}\cdot \sigma k_x.
\end{equation}
 From \eqref{1.1}, we have that
\begin{equation}\label{3.11}
\left\{\begin{array}{ll}
\displaystyle  u_x=\frac{1}{u^2-c^2}\left(\frac{c^2}{A}A_t-uu_t-u\left(\frac{A^m}{\rho A_e^{m}}-\frac{1}{\rho}\right)k_x\right),\\[12pt]
\displaystyle  A_x=\frac{1}{u^2-c^2}\cdot\left(Au_t-uA_t\right)+\frac{A\left(\left(\frac{A}{A_e}\right)^{m}-1\right)}{\rho(u^2-c^2)}k_x.
\end{array}\right.
\end{equation}
We take the limit $t\to 0+$ in \eqref{3.11} by using \eqref{3.5} and \eqref{3.6}. Then we substitute \eqref{3.11} into \eqref{3.10}, which follows  \eqref{3.7}.
\end{proof}
\end{lem}

Combining the above discussion, we now solve the generalized Riemann problem \eqref{1.3} and \eqref{2.6} with the setup of Fig. 2.1, i.e., the rarefaction wave moves to the left and the shock wave moves to the right. We summarize the results in the following propositions.
\begin{prop}
(Nonsonic case) Assume that the $t$-axis is not in the rarefaction wave, then the limiting values $\left(\frac{\partial u}{\partial t}\right)_{*}$ and $\left(\frac{\partial A}{\partial t}\right)_{*}$ are obtained by solving the following pair of linear equations
\begin{equation}\label{3.12}
\left\{\begin{array}{ll}
\displaystyle a_L\left(\frac{\partial u}{\partial t}\right)_{*}+b_L\left(\frac{\partial A}{\partial t}\right)_{*}=d_L,\\[12pt]
\displaystyle a_R\left(\frac{\partial u}{\partial t}\right)_{*}+b_R\left(\frac{\partial A}{\partial t}\right)_{*}=d_R,
\end{array}\right.
\end{equation}
where $a_L,a_R,b_L,b_R,d_L$ and $d_R$ are defined in  Lemma 3.1 and Lemma 4.1. Appendix A gives the other cases. These coefficients depend only on the initial data \eqref{2.6} and the local Riemann solution $R(0;U_L,U_R)$.
\end{prop}

When the $t$-axis is also a characteristics in the rarefaction wave, we have a sonic case, the result is shown in the following proposition.
\begin{prop}
(Sonic case) Assume that the $t-$axis is in the rarefaction wave associated with $u-c$ characteristic field. Then we have
\begin{equation}\label{3.13}
\left\{\begin{array}{ll}
\displaystyle  \left(\frac{\partial u}{\partial t}\right)_{*}=\frac12\left(d_L(0)+k'(0)\cdot\left(\frac1{\rho}-\frac{u_{*}c_{*}}{mk(0)}\right)\right),\\[10pt]
\displaystyle  \left(\frac{\partial A}{\partial t}\right)_{*}=\frac{A_{*}}{2c_{*}}\left(d_L(0)-k'(0)\cdot\left(\frac1{\rho}-\frac{u_{*}c_{*}}{mk(0)}\right)\right),
\end{array}\right.
\end{equation}
where $d_L(0)$ is defined in Lemma 3.1.
\end{prop}
\begin{proof}
On the one hand, from \eqref{2.3}, we have that
\begin{equation}\label{3.14}
\frac{\partial \psi}{\partial t}=\frac{\partial u}{\partial t}+\frac{c}{A}\cdot \frac{\partial A}{\partial t}, \quad \frac{\partial \phi}{\partial t}=\frac{\partial u}{\partial t}-\frac{c}{A}\cdot \frac{\partial A}{\partial t}.
\end{equation}
Since in the sonic case, we have $\beta=u_{*}-c_{*}\to 0$ as $t\to 0_{+}$, so $\left(\partial \psi/\partial t\right)_{*}=d_L(0)$. Thus we get
\begin{equation}\label{3.15}
\left(\frac{\partial u}{\partial t}\right)_{*}+\frac{c_{*}}{A_{*}} \cdot\left( \frac{\partial A}{\partial t}\right)_{*}=d_L(0).
\end{equation}
On the other hand, at the origin, we have
\begin{equation}\label{3.16}
\left(\frac{\partial \phi}{\partial t}\right)_{*}=\left(\frac{\partial \phi}{\partial t}\right)_{*}+(u_{*}-c_{*})\cdot \left(\frac{\partial \phi}{\partial x}\right)_{*}=\left(\frac1{\rho}-\frac{u_{*}c_{*}}{mk(0)}\right)k'(0).
\end{equation}
We combine \eqref{3.14}-\eqref{3.16}, and finally obtain \eqref{3.13}.
\end{proof}

\section{The acoustic case}
In this section, we consider a special case: the acoustic case. Assume that the initial data $U_L=U_R$, but we allow jumps in their slopes $U'_L\neq U'_R$. Obviously, the associated Riemann solution is now constant:
\begin{equation}\label{5.1'}
\displaystyle U(x,t)=U_L=U_R.
\end{equation}
Hence the nonlinear waves degenerate to a characteristic curve whose two side states are the same. The acoustic case can be viewed as a proper linearization of the nonlinear system \cite{BenL1}. We have the following lemma.
\begin{lem}
When $U_L=U_R$ and $U'_L\neq U'_R$, we have the acoustic case. $(\partial u/\partial t)_{*}$ and $(\partial A/\partial t)_{*}$ are given by
\begin{equation}\label{5.1}
\left\{\begin{array}{ll}
\displaystyle  \left(\frac{\partial u}{\partial t}\right)_{*}=\displaystyle -\frac{u_*}{2A_*}\left[A_Lu'_L+A_Ru'_R+c_*(A'_L-A'_R)\right]-\frac{c_*}{2A_*}\left[A_Lu'_L-A_Ru'_R+c_*(A'_L+A'_R)\right]\\[12pt]
\qquad \qquad  \displaystyle -\frac{k'(0)}{\rho}\left(\left(\frac{A_*}{A_e}\right)^{m}-1\right),\\[12pt]
\displaystyle  \left(\frac{\partial A}{\partial t}\right)_{*}=\displaystyle -\frac12\left[A_Lu'_L+A_Ru'_R+c_*(A'_L-A'_R)\right]-\frac{u_*}{2c_*}\left[A_Lu'_L-A_Ru'_R+c_*(A'_L+A'_R)\right].
\end{array}\right.
\end{equation}
\end{lem}
\begin{proof}
In the acoustic case, the nonlinear wave degenerates to a characteristic curve. Denote $\Gamma_-$ ($\Gamma_+$) as $u-c$ ($u+c$) characteristic curve. Then the solution is continuous across $\Gamma_-$ and $\Gamma_+$. First, we resolve the states across $\Gamma_-$. Denote $U_-(x,t)$ and $U_1(x,t)$ as the left hand side and the right hand side of $\Gamma_-$, respectively.  We take the differentiation along it for the variable $A$, and get
\begin{equation}\label{5.2}
\displaystyle  \frac{\partial A_-}{\partial t}+(u-c)\frac{\partial A_-}{\partial x}= \frac{\partial A_1}{\partial t}+(u-c)\frac{\partial A_1}{\partial x}.
\end{equation}
By using \eqref{1.3}, we have
\begin{equation}\label{5.3}
-A_-(u_-)_x-c(A_-)_x=-Au_x-cA_x.
\end{equation}
Taking the limit $t\to 0_{+}$ in \eqref{5.3}, we have
\begin{equation}\label{5.4}
-A_Lu'_L-c_*A'_L=-A_*u_x-c_*A_x.
\end{equation}
Similarly, by resolving the acoustic wave along the characteristic wave $\Gamma_+$, we have
\begin{equation}\label{5.5}
-A_Ru'_R+c_*A'_R=-A_*u_x+c_*A_x.
\end{equation}
By solving \eqref{5.4} and \eqref{5.5}, we have that
\begin{equation}\label{5.6}
\left\{\begin{array}{ll}
\displaystyle  A_x=\frac{1}{2c_*}\left(A'_Lu'_L-A_Ru'_R+c_*(A'_L+A'_R)\right),\\[9pt]
\displaystyle  u_x=\frac{1}{2A_*}\left(A_Lu'_L+A_Ru'_R+c_*(A'_L-A'_R)\right).\\
\end{array}\right.
\end{equation}
We use \eqref{1.3} again to get
\begin{equation}\label{5.7}
\left\{\begin{array}{ll}
\displaystyle  A_t=-Au_x-uA_x,\\[9pt]
\displaystyle  u_t=-uu_x-\frac{k_x}{\rho}\left(\left(\frac{A}{A_e}\right)^{m}-1\right)-\frac{c^2}{A}A_x.
\end{array}\right.
\end{equation}
By substituting \eqref{5.6} into \eqref{5.7}, we finally obtain \eqref{5.1}. Thus we prove the lemma.
\end{proof}

\section{Two-dimensional extension}
In this section, we extend the GRP to the two-dimensional blood flow equations. When we consider the blood flow in the excess of fats, cholesterol plaques and blood clots, or in some blood diseases like polycythemia,  which refers to the blood flow in porous medium, it is better to use two dimensional blood flow model. See \cite{Ullah} for more details. For simplicity, we assume that the coefficient $k(x)$ is constant here. Then the extended two-dimensional model is given by
\begin{equation}\label{6.1}
\left\{\begin{array}{l}
A_t+(A u)_x+(Av)_y=0,\\[3pt]
\displaystyle (A u)_t+(Au^2+\overline{p})_x+(Auv)_y=0,\\[3pt]
\displaystyle (A v)_t+(Auv)_x+(Av^2+\overline{p})_y=0.
\end{array}\right.
\end{equation}
Here
\begin{equation}\label{6.2}
 \overline{p}=\frac{mk}{\rho(m+1)}\frac{A^{m+1}}{A_e^{m}}.
\end{equation}
Following \cite{BenL1}, we use the Strang splitting method (see \cite{Strang}) to split \eqref{6.1} into two one dimensional subsystems,
\begin{equation}\label{6.3}
\begin{array}{lll}
\left\{\begin{array}{ll}
A_t+(Au)_x=0,\\[3pt]
(Au)_t+(Au^2+\overline{p})_x=0, \\[3pt]
(Av)_t+(Auv)_x=0, \end{array} \right.
\ \text{and}\
\left\{\begin{array}{ll}
A_t+(Av)_y=0,\\[3pt]
(Au)_t+(Auv)_y=0, \\[3pt]
(Av)_t+(Av^2+\overline{p})_y=0. \end{array} \right.
\end{array}
\end{equation}
The 2D Strang splitting method is given by
\begin{equation}\label{6.4}
U^{n+1}=\mathcal{L}_x\left(\frac{\Delta t}{2}\right)\mathcal{L}_y\left(\Delta t\right)\mathcal{L}_x\left(\frac{\Delta t}{2}\right)U^{n}.
\end{equation}
Thus we only need to resolve the GRP for the velocity component $v$. Note that $\partial v/\partial t+u\cdot\partial v/\partial t=0$ from \eqref{6.3}.  We  have  the following lemma.
\begin{lem}
Assume that a rarefaction wave moves to the left and a shock wave moves to the right, $t-$axis is in the intermediate region, see Fig. 2.1. Then

1). If $u_*>0$, the value $(\partial v/\partial t)_{*}$ is computed from the rarefaction wave side by:
\begin{equation}\label{6.5}
\left(\frac{\partial v}{\partial t}\right)_{*}=-u_{*}\cdot \frac{\rho_{*}}{\rho_L}\cdot v'_{L}.
\end{equation}

2). If $u_*<0$, the value $(\partial v/\partial t)_{*}$ is computed from the shock wave side by:
\begin{equation}\label{6.6}
\displaystyle \left(\frac{\partial v}{\partial t}\right)_{*}=-\frac{u_{*}(\sigma_R-u_R)}{\sigma_R-u_*}v'_R.
\end{equation}
\end{lem}
For the proof of lemma 6.1, we refer to \cite{BenL1} and not repeat here.

\section{Implementation of the GRP scheme}
In this section we  describe the implementation of GRP scheme by the following four steps.

\noindent
{\bf Step 1.} Given piecewise linear initial data
\begin{equation}\label{7.1}
U^{n}(x)=U_j^{n}+\sigma_j^{n}(x-x_j),\quad x\in (x_{j-1/2},x_{j+1/2}),
\end{equation}
we solve the Riemann problem of \eqref{2.7} to define the Riemann solution
\begin{equation}\label{7.2}
U_{j+1/2}^{n}=R^{A}\left(0;U_j^{n}+\frac{\Delta x}{2}\sigma_j^{n},U_{j+1}^{n}-\frac{\Delta x}{2}\sigma_{j+1}^{n}\right).
\end{equation}

\noindent
{\bf Step 2.} Determine $\left(\partial U/\partial t\right)_{j+1/2}^{n}$. This is obtained by solving a linear algebraic equations according to Lemma 3.1, Lemma 4.1 and Lemma 5.1. The other cases are summarized in Appendix A.

\noindent
{\bf Step 3.} Evaluate the new cell averages $U_j^{n+1}$ using \eqref{1.6} and \eqref{1.8}.

\noindent
{\bf Step 4.} Update the slope by the following procedure,
\begin{equation}\label{7.3}
\left\{\begin{array}{ll}
\displaystyle U_{j+1/2}^{n+1,-}=U_{j+1/2}^{n}+\Delta t\left(\frac{\partial U}{\partial t}\right)_{j+1/2}^{n},\\[12pt]
\displaystyle \sigma_{j}^{n+1,-}=\frac1{\Delta x}(\Delta U)_{j}^{n+1,-}:=\frac1{\Delta x}\left(U_{j+1/2}^{n+1,-}-U_{j-1/2}^{n+1,-}\right).
\end{array}\right.
\end{equation}
In order to suppress local oscillations near discontinuities, we modify $\sigma_{j}^{n+1,-}$ by a monotonicity algorithm to get $\sigma_{j}^{n+1}$, see \cite{BenF1,BenL2},
\begin{equation}\label{7.4}
\sigma_j^{n+1}={\rm minmod}\left(\alpha \frac{U_j^{n+1}-U_{j-1}^{n+1}}{\Delta x}, \sigma_j^{n+1,-}, \alpha \frac{U_{j+1}^{n+1}-U_j^{n+1}}{\Delta x}\right).
\end{equation}
In section 8, we take $\alpha=0.9$ for our tests.

\section{Numerical examples}
In this section we give some numerical examples to study the convergence rate and test the effectiveness of the GRP  scheme for blood flow model \eqref{1.1}. We also compare the results with the second order Godunov scheme. The average blood  density is taken as
$\rho=1.05 g/cm^{3}$, the equilibrium state $A_e=1.0 cm^{2}$. We note here that the unit for vessel diameter is $cm^{2}$ and the velocity is $cm/s$. For the coefficient $k(x)$,  we take the unit as $g/cm/s^{2}$. The CFL constant is 0.5 here.

\subsection{\emph{One-dimensional examples}} \

\noindent
{\bf Example 1} (Empirical convergence rates).  We first carry out an empirical convergence rate study of the derived GRP scheme. In \cite{Muller2}, the authors implement the convergence rate studies on well-balanced higher order schemes using the method of manufactured solutions. Here follow \cite{Muller2}, we manufacture an exact reference solution of system (1.1) which is given by 
\begin{equation}\label{8.1}
\hat{A}(x,t)=A_0+a_0sin(2\frac{\pi}{L}x)cos(2\frac{\pi}{T_0}t)
\end{equation}
for cross-sectional area $A(x,t)$ and 
\begin{equation}\label{8.2}
\hat{q}(x,t)=q_0-\frac{a_0L}{T_0}cos(2\frac{\pi}{L}x)sin(2\frac{\pi}{T_0}t)
\end{equation}
for the mass flux $q(x,t)=A(x,t)u(x,t)$. $k(x)$ in \eqref{1.2} is described by 
 \begin{equation}\label{8.3}
\hat{k}(x)=K_0+k_0sin(2\frac{\pi}{L}x).
\end{equation}
As shown in \cite{Muller2}, \eqref{8.1}-\eqref{8.3} are not solutions of the original system, but are exact solutions of the following modified system 
\begin{equation}\label{8.4}
{\rm \partial}_t U+{\rm \partial}_x F(U)+L(U){\rm \partial}_x G(x)=\widetilde{B}(x,t)
\end{equation}
with an extra source term $\widetilde{B}(x,t)$. System \eqref{8.4} is used to carry out the convergence rate study.

\begin{figure}[htbp]
\centering
\includegraphics[width=0.65\textwidth]{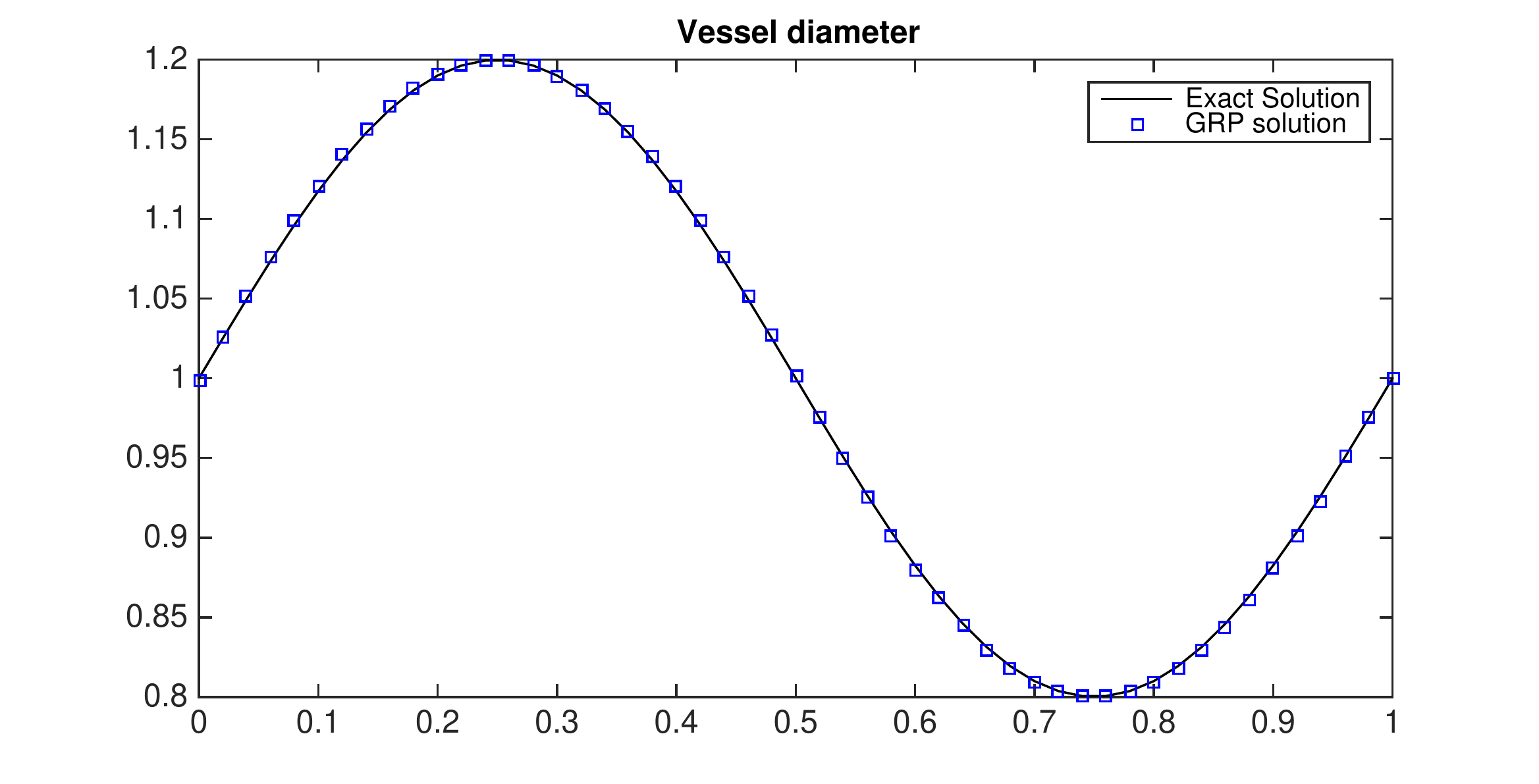}
\caption*{Fig. 8.1. The GRP solution of vessel diameter in the convergence rate study with 51 cells at time t=0.025s. The parameter $K_0=6.0$.}
\end{figure}
\begin{table*}[]
\scriptsize
\centering
\caption{The $L^{\infty}$ errors of vessel diameter $A(x,t)$ and empirical convergence rates for the GRP solver.}
\label{Tab03}
\begin{tabular}{cccccccccc} 
\toprule
 \multirow{2}{*}{$K_0$}& \multirow{2}{*}{$t_0$$(s)$}  & \multicolumn{2}{c}{$t=t_0$}  &  \multicolumn{2}{c}{$t=t_0/2$}  & \multicolumn{2}{c}{$t=t_0/4$} & \multicolumn{2}{c}{$t=t_0/8$}   \\
\cmidrule(r){3-4} \cmidrule(r){5-6} \cmidrule(r){7-8} \cmidrule(r){9-10}
& &  ${\rm Error}$ \quad  \quad &   ${\rm Order}$ \quad \quad
&  ${\rm Error}$  \quad \quad &   ${\rm Order}$ \quad \quad
&  ${\rm Error}$  \quad \quad &   ${\rm Order}$ \quad \quad
&  ${\rm Error}$ \quad \quad  &   ${\rm Order}$  \\
\midrule
$6.0 $     \qquad      & 0.1                \qquad             & 5.632e-02            \qquad          & --           \qquad          & 1.296e-02         \qquad     & 2.12      \qquad      & 3.068e-03         \qquad     & 2.08    \qquad     &7.465e-04 \qquad     & 2.04     \\ [3pt]
$60.0 $        \qquad        &0.1                 \qquad            & 2.276e-01           \qquad            &--         \qquad             & 5.356e-02      \qquad       & 2.09    \qquad        & 1.272e-02       \qquad       & 2.07     \qquad     & 3.105e-03 \qquad     & 2.03      \\  [3pt]
$300.0 $       \qquad       &0.02                \qquad     & 1.289e-01          \qquad           & --        \qquad             &3.103e-02      \qquad     & 2.05          & 7.755e-03        \qquad      & 2.00       \qquad     & 1.937e-03 \qquad     & 2.00     \\  [3pt]
$600.0 $       \qquad       &0.01              \qquad              & 6.252e-02            \qquad          & --            \qquad          & 1.533e-02     \qquad       &  2.03        \qquad     & 3.681e-03        \qquad      & 2.06    \qquad     & 9.082e-04 \qquad     & 2.02     \\
\bottomrule
\end{tabular}
\end{table*}
We test the GRP solvers with the initial coefficient $K_0$ ranging from 6 to 600, which represents the blood pressure ranging from low to moderate, respectively.  We use 51 cells in all tests. The other parameters in the problem are: $A_0$ = 1.0$ {\rm cm^2}$, $a_0$ = 0.2 ${\rm cm^2}$, $L$ = 1.0 cm, $T_0$ = 3s, $q_0 = 0 \frac{{\rm cm^3}}{s}$, $k_0$ = 1.2.

In Fig. 8.1. we can see the GRP solution fits well with the exact solution. Table 1 shows the empirical convergence rates for the GRP solver. The error is measured in $L^{\infty}$ norm. From the table we see the performance of GRP solver for small pressure ($K_0$= 6.0) does better than that for moderate pressure ($K_0$= 600). We also see that for all cases the GRP scheme attains order two, which is the expected order for the GRP solvers in the blood flow model.

\noindent
{\bf Example 2} (Riemann problem 1). We consider the Riemann problem with $k(x)$ being constant first. We take the initial data with
\begin{equation}\label{8.5}
\begin{array}{ll}
(A(x,0), u(x,0))=\left\{ \begin{array}{cc}
(3.5, 3.5), \quad &0\leq x< 0.3,\\[2pt]
(2.5, 5.0), \quad &0.3< x\leq 10.
\end{array}
\right.
\end{array}
\end{equation}
Here $k(x)=10$.  Numerical results are shown at time $t=0.75s$. See Fig. 8.2.  The solution includes a backward rarefaction wave, followed by a forward rarefaction wave. In Table 2, we  compare the $L^{1}$ error of  GRP scheme with that of Godunov scheme. We see that GRP scheme does better than Godunov scheme and GRP scheme also converges very fast for small grid sizes.
\begin{table}[htp]
\caption*{Table 2. Comparison of $L^1$ errors for example 2.}
\begin{center}
\begin{tabular}{|c|c|c|c|}
\hline
Number of nodes & GRP $L^{1}$-error & Godunov $L^{1}$-error \\
\hline
100 & 0.0464 & 0.0591\\
\hline
200 & 0.0374 & 0.0386 \\
\hline
300 & 0.0183 & 0.0278\\
\hline
400 & 0.0144 & 0.0225\\
\hline
\end{tabular}
\end{center}
\label{default}
\end{table}%

\begin{figure}[htbp]
\subfigure{
\begin{minipage}[t]{0.48\textwidth}
\centering
\includegraphics[width=1.1\textwidth]{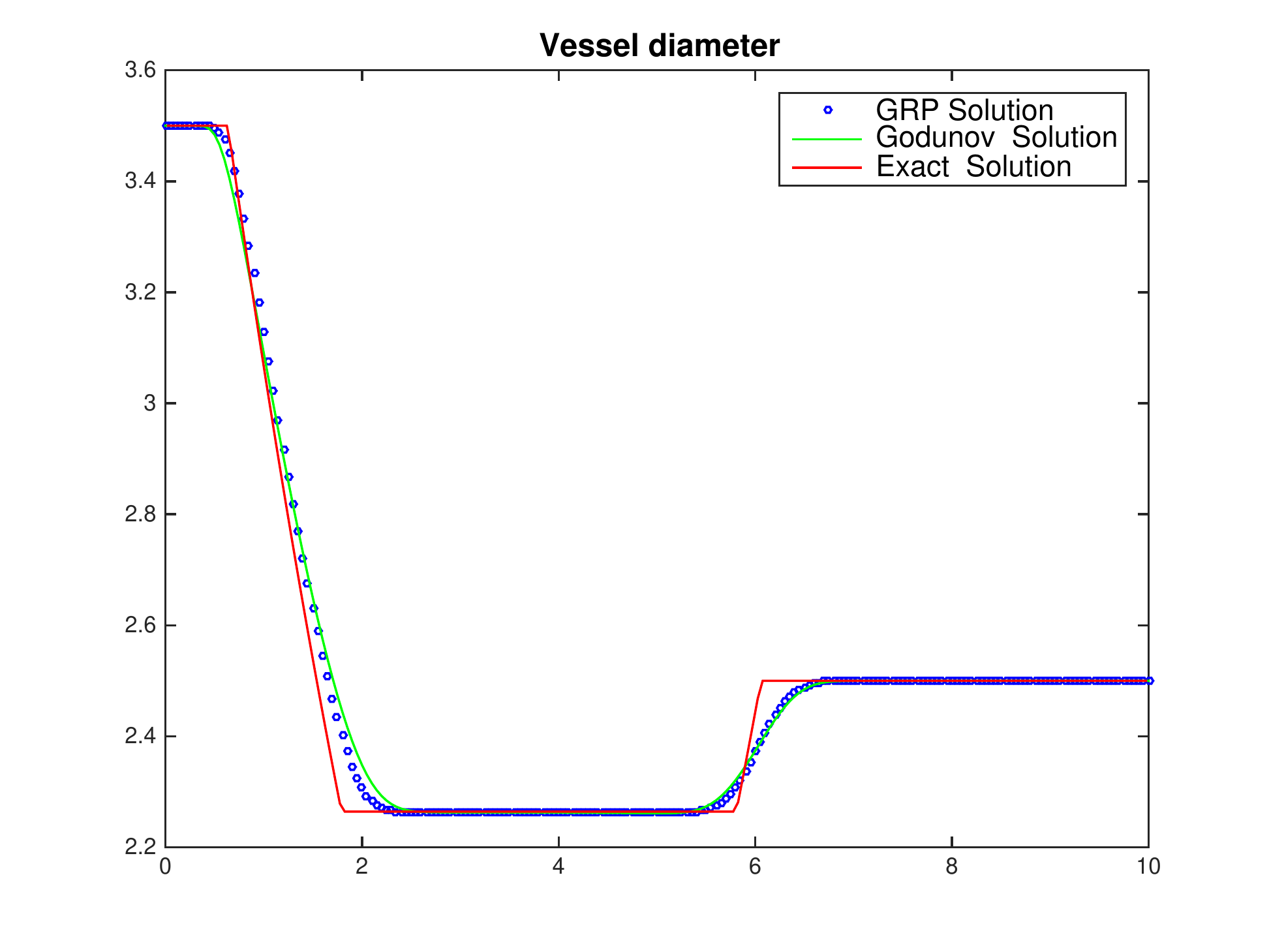}
\end{minipage}
}
\subfigure{
\begin{minipage}[t]{0.48\textwidth}
\centering
\includegraphics[width=1.1\textwidth]{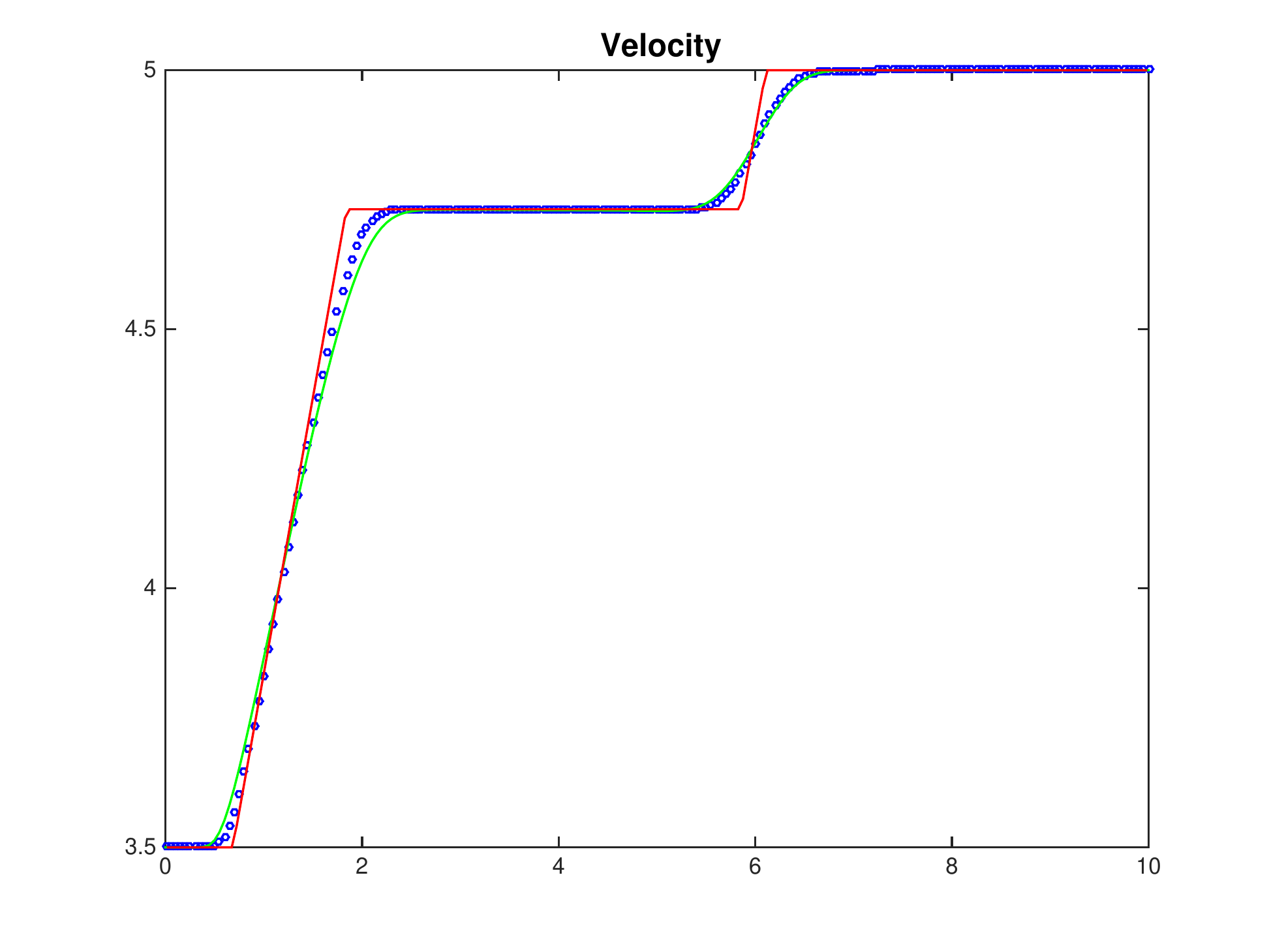}
\end{minipage}
}
\caption*{Fig. 8.2. Numerical results for Riemann problem 1: 200 grids are used.}
\end{figure}

\noindent
{\bf Example 3} (Riemann problem 2). We consider the Riemann problem where $k(x)$ changes continuously in this case. We take the initial data with
\begin{equation}\label{8.6}
\begin{array}{ll}
(A(x,0), u(x,0))=\left\{ \begin{array}{cc}
(1.2, 3.5), \quad &0\leq x< 0.6,\\[2pt]
(1.5, 2.5), \quad &0.6< x\leq 4.0.
\end{array}
\right.
\end{array}
\end{equation}

 $k(x)$ is given by
\begin{equation}\label{8.7}
\begin{array}{ll}
k(x)=\left\{ \begin{array}{cc}
6.0,\quad &0\leq x< 0.6, \\[2pt]
6.0\cdot \left(1.0-0.5*{\rm sin}(2.5\pi\cdot (x-0.6))\right), \quad &0.6\leq x< 0.8,\\[2pt]
3.0, \quad &0.8\leq x\leq 4.0.
\end{array}
\right.
\end{array}
\end{equation}

Numerical results are shown at time $t=0.6s$. See Fig. 8.3.  The solution includes a stationary wave first, followed by a backward shock wave, then followed by a forward shock wave. In Table 3, we present the $L^{1}$ error of GRP scheme. It can be seen that GRP scheme captures the steady wave quite well.

\begin{table}[htp]
\caption*{Table 3. Comparison of $L^1$ errors for example 3.}
\begin{center}
\begin{tabular}{|c|c|c|c|}
\hline
Number of nodes & GRP $L^{1}$-error & Godunov $L^{1}$-error \\
\hline
100 &  0.0852 & 0.1120 \\
\hline
200 &   0.0651 & 0.0712 \\
\hline
300 & 0.0328  & 0.0523 \\
\hline
400 &  0.0236 & 0.0412 \\
\hline
\end{tabular}
\end{center}
\label{default}
\end{table}%

\begin{figure}[htbp]
\subfigure{
\begin{minipage}[t]{0.48\textwidth}
\centering
\includegraphics[width=1.1\textwidth]{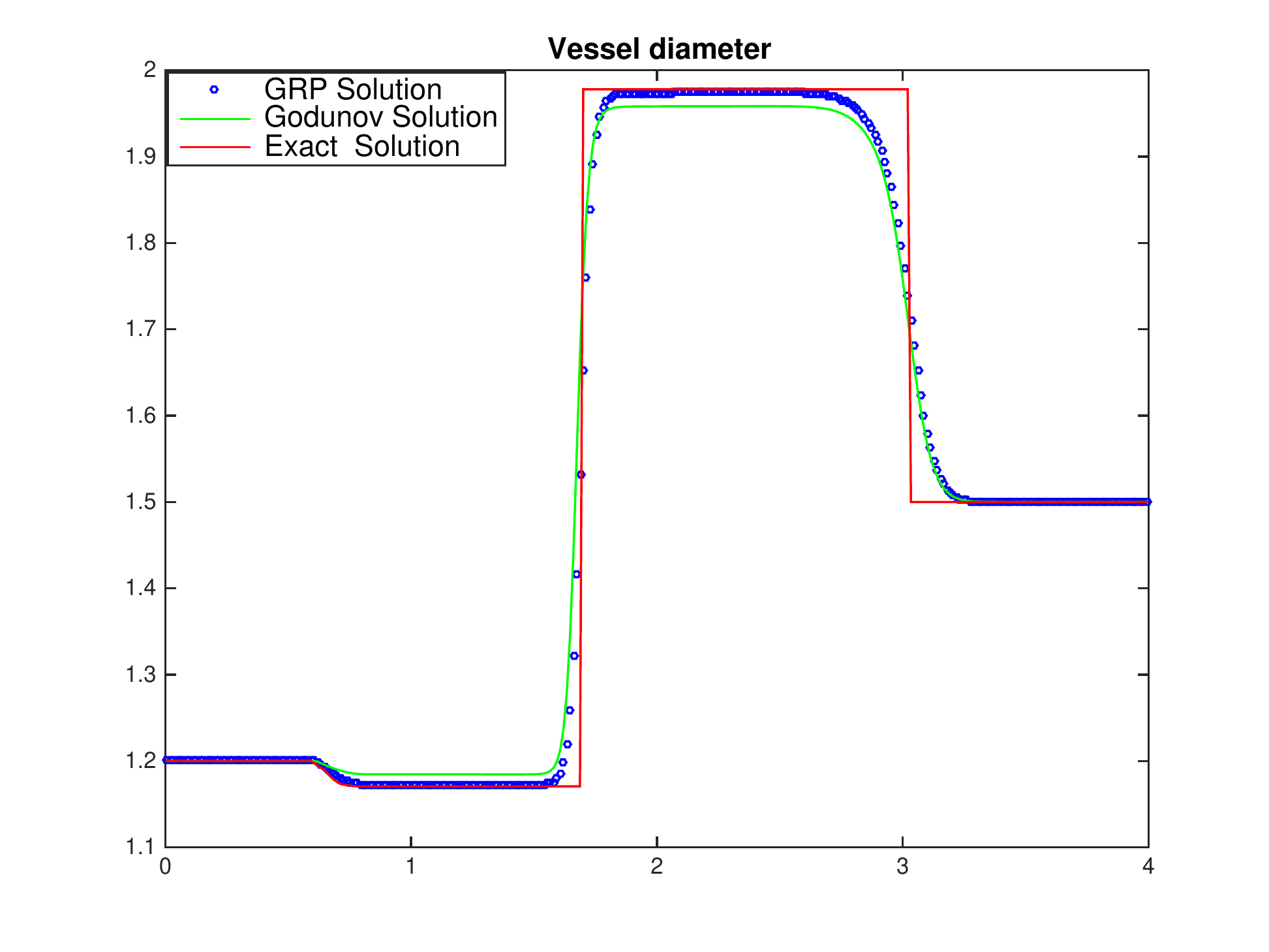}
\end{minipage}
}
\subfigure{
\begin{minipage}[t]{0.48\textwidth}
\centering
\includegraphics[width=1.1\textwidth]{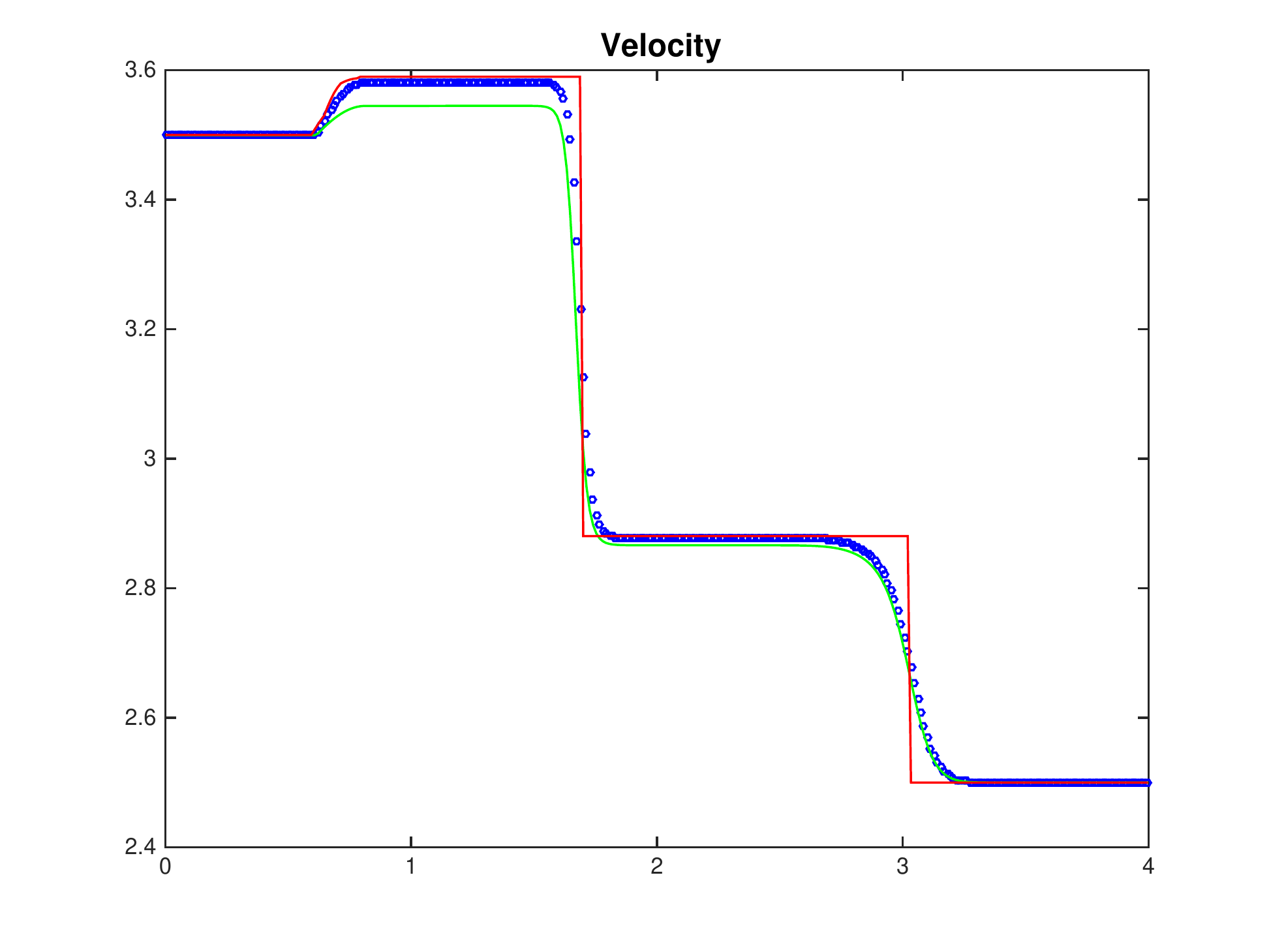}
\end{minipage}
}
\caption*{Fig. 8.3. Numerical results for Riemann problem 2: 300 grids are used.}
\end{figure}

\noindent
{\bf Example 4} (Shock interaction problem). We consider the shock interaction problem in this case. The initial data is equipped with
\begin{equation}\label{8.8}
\begin{array}{ll}
(A(x,0), u(x,0))=\left\{ \begin{array}{cc}
(2.0, 3.061),\quad & 0\leq x< 0.6,\\[2pt]
(1.5, 2.5), \quad & 0.6\leq x< 0.8,\\[2pt]
(1.3307492, 2.818), \quad & 0.8\leq x\leq 15.0,
\end{array}
\right.
\end{array}
\end{equation}
and $k(x)$ is given by
\begin{equation}\label{8.9}
\begin{array}{ll}
k(x)=\left\{ \begin{array}{cc}
6.0,\quad &0\leq x<0.6, \\[2pt]
6.0\cdot \left(1.0-0.5*{\rm sin}(2.5\pi\cdot (x-0.6))\right), \quad &0.6\leq x<0.8,\\[2pt]
3.0, \quad &0.8\leq x\leq 15.0.
\end{array}
\right.
\end{array}
\end{equation}
A shock wave emanates from $x=0.6$ at time $t=0$ and propagates toward right. Which interacts with the stationary wave located in the interval $0.6\leq x\leq 0.8$.
Numerical results are shown at time $t=1.5s$. See Fig. 8.4. From Table 4, we conclude that the GRP scheme converges faster than the Godunov scheme. Moreover, when we shrink grid sizes such that the node number is greater than $400$, the $L^1$ error of the GRP scheme dose not decay as much as that in the small grid size case, which has a similar behavior as the Godunov scheme.

\begin{table}[htp]
\caption*{Table 4. Comparison of $L^1$ errors for example 4.}
\begin{center}
\begin{tabular}{|c|c|c|c|}
\hline
Number of nodes & GRP $L^{1}$-error & Godunov $L^{1}$-error \\
\hline
200 &  0.0712 & 0.1031\\
\hline
400 &  0.0275 & 0.0582  \\
\hline
800 & 0.0163  & 0.0359  \\
\hline
1000 & 0.0127  & 0.0306 \\
\hline
\end{tabular}
\end{center}
\label{default}
\end{table}%

\begin{figure}[htbp]
\subfigure{
\begin{minipage}[t]{0.48\textwidth}
\centering
\includegraphics[width=1.1\textwidth]{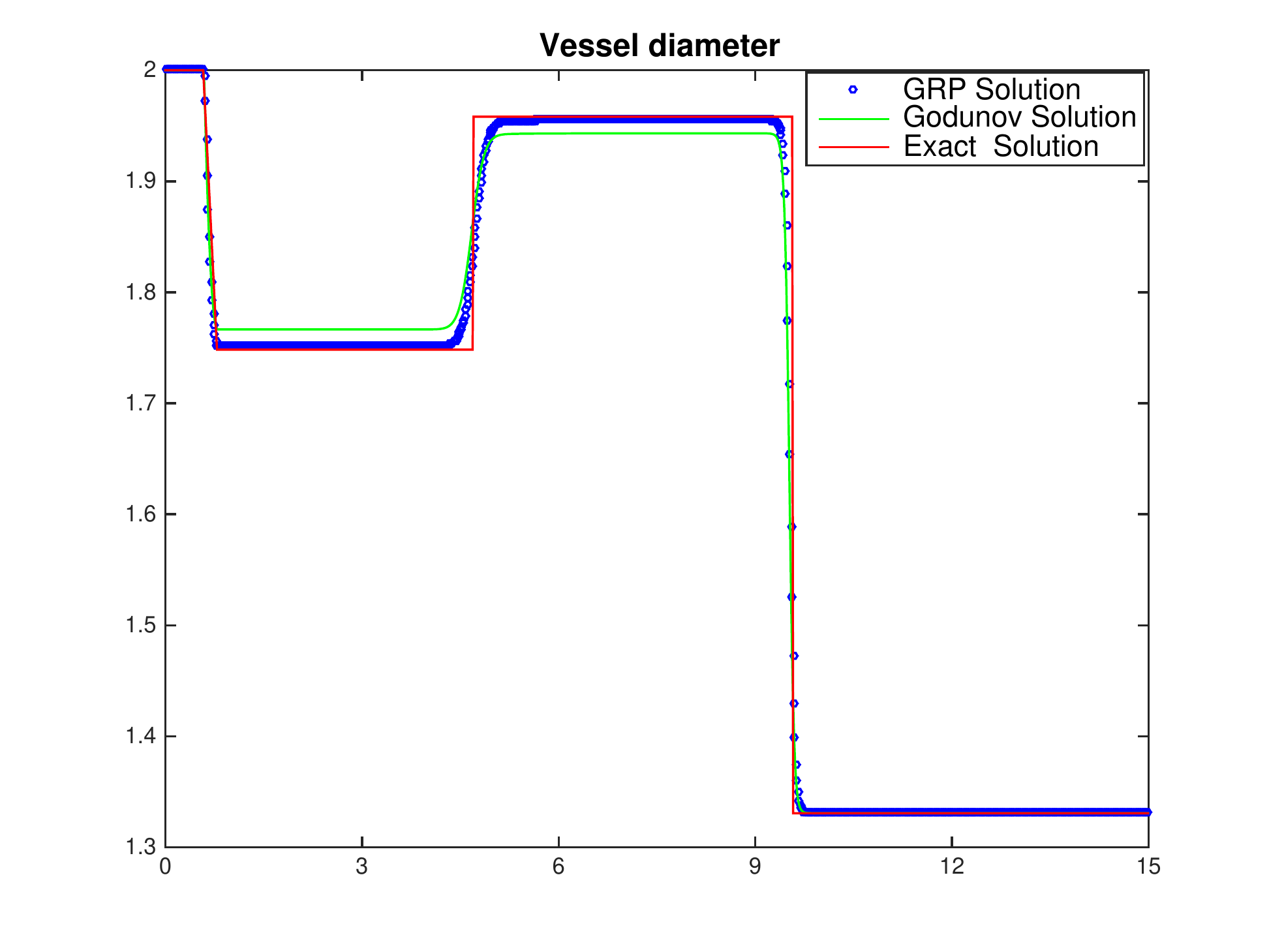}
\end{minipage}
}
\subfigure{
\begin{minipage}[t]{0.48\textwidth}
\centering
\includegraphics[width=1.1\textwidth]{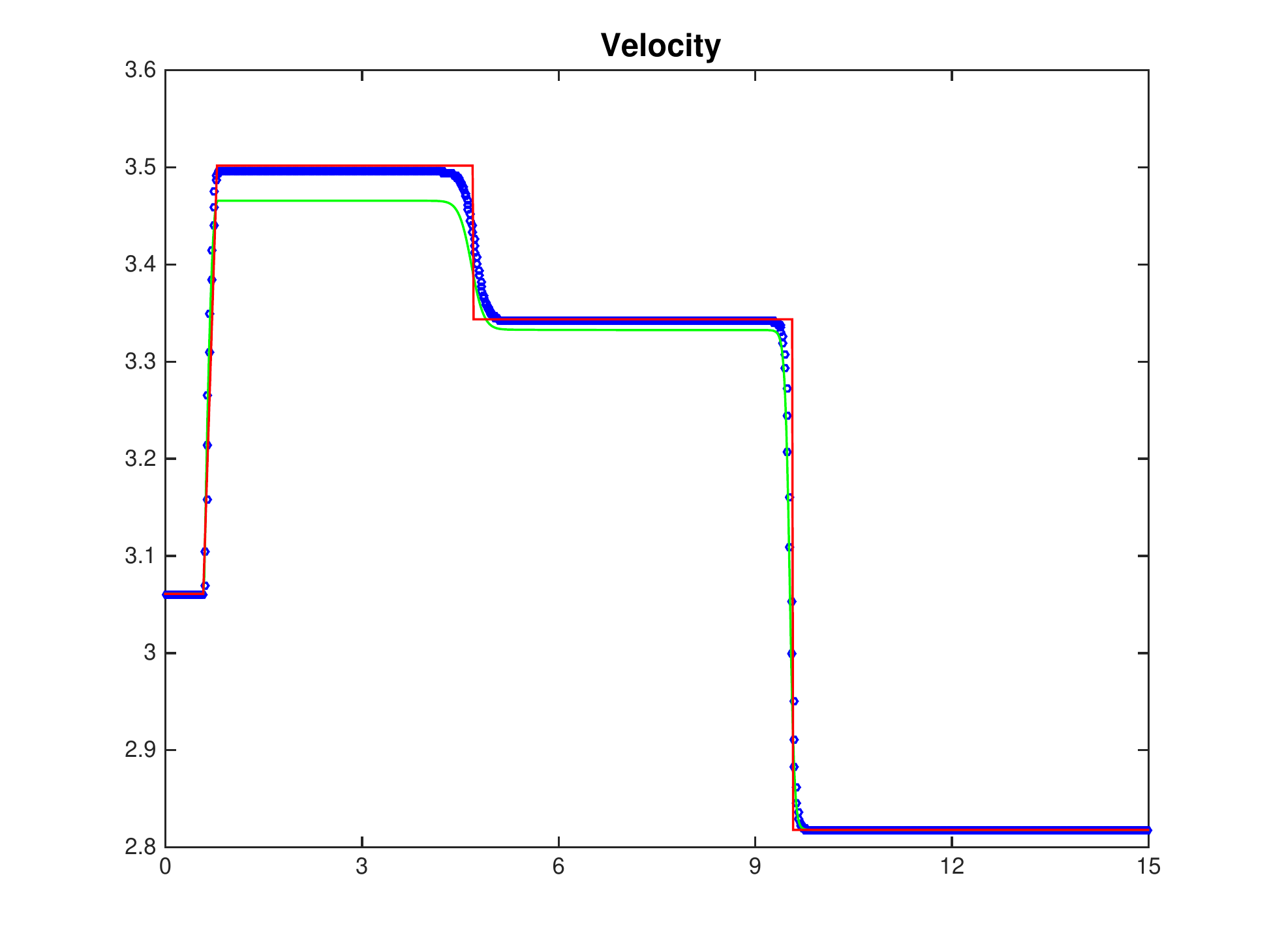}
\end{minipage}
}
\caption*{Fig. 8.4. Numerical results for shock interaction problem.}
\end{figure}

\subsection{\emph{Two-dimensional examples}} \

We give some numerical examples of two-dimensional Riemann problem for blood flow model \eqref{6.1}.  The two-dimensional Riemann problems for Euler equations were proposed by Zhang and Zheng \cite{Zhang}. Some numerical examples can be found in \cite{BenL2,Chang,Zheng}. Since the structure of Riemann solutions of blood flow is similar to Euler equations, we use the two-dimensional Riemann problems to checking the accuracy of GRP scheme. Our results are new and provide a direction of multidimensional blood flow research. Each example consists of four constant states in the four quadrants. We take $k=5.0$ in the following examples.

\noindent
{\bf Example 5} (Interaction of rarefaction waves).  The Riemann initial data is chosen as:
\begin{equation}\label{8.10}
\begin{array}{llll}
(A(x,y,0), u(x,y,0),v(x,y,0))=\left\{ \begin{array}{cc}
(1.0, 0.0, 0.0),\quad &  x> 0.5,~y> 0.5,\\[4pt]
(0.5179, -0.9316, 0.0),\quad &  x< 0.5,~y> 0.5,\\[4pt]
(0.1492, -0.9316, -1.4045),\quad & x< 0.5,~y< 0.5,\\[4pt]
(0.356, 0.0, -1.4045), \quad & x>0.5, y<0.5.
\end{array}
\right.
\end{array}
\end{equation}

\begin{figure}[htbp]
\centering
\includegraphics[width=0.75\textwidth]{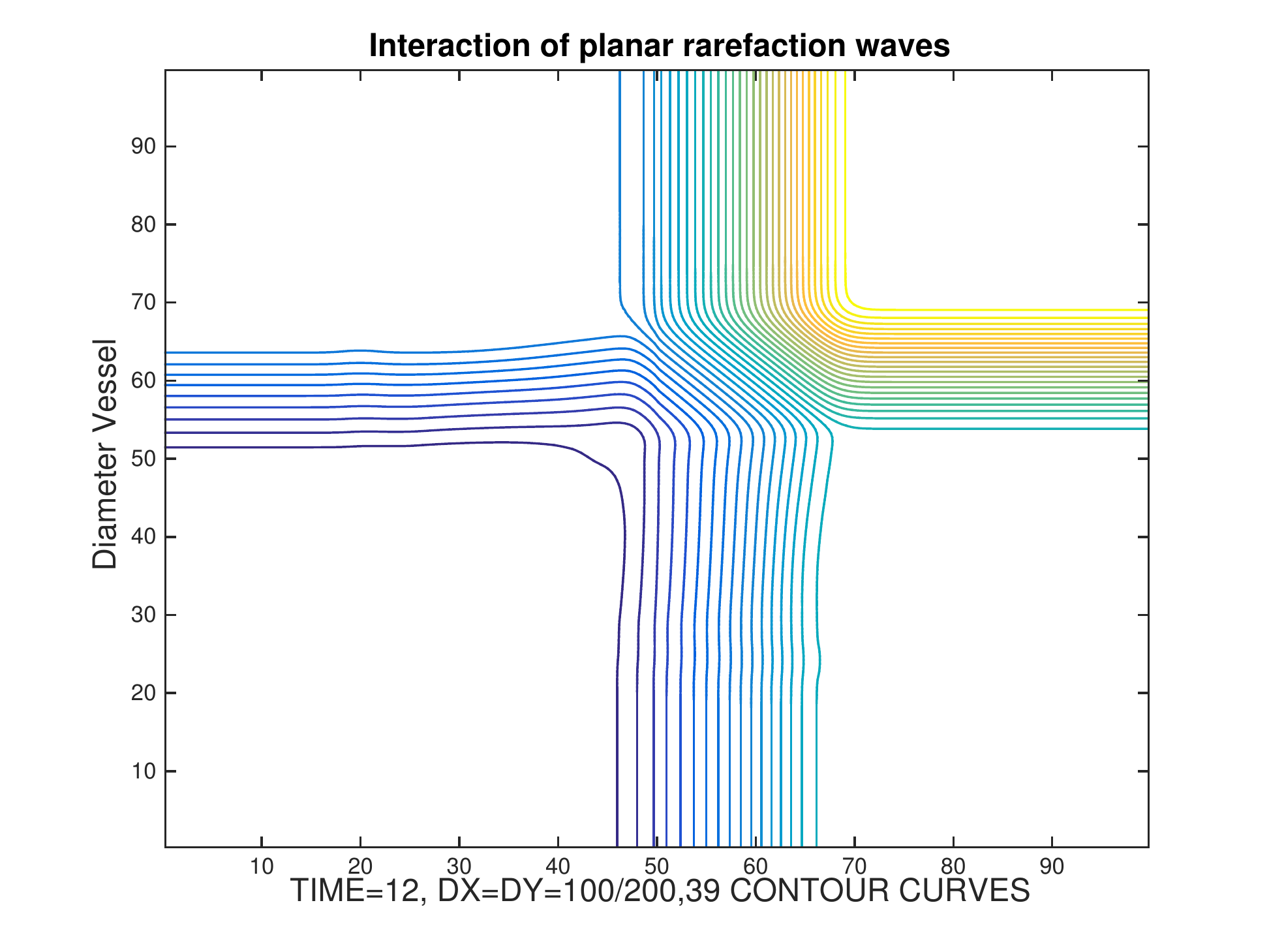}
\caption*{Fig. 8.5. Numerical results for four planar rarefaction waves.}
\end{figure}
We see four planar rarefaction waves interact with each other, they penetrate each other and there is no new types of wave patterns in this case. See Fig. 8.5. Our result is similar compared to the Euler cases in \cite{Chang}.

\noindent
{\bf Example 6} (The formation of shocks in continuous domain).  We still consider the interaction of four rarefaction waves.
The Riemann initial data is chosen as:
\begin{equation}\label{8.11}
\begin{array}{llll}
(A(x,y,0), u(x,y,0),v(x,y,0))=\left\{ \begin{array}{cc}
(1.0, 0.0, 0.0),\quad &  x> 0.5,~y> 0.5,\\[2pt]
(1.552, -0.7169, 0.0),\quad &  x< 0.5,~y> 0.5,\\[2pt]
(1.0, -0.7169, -0.7169),\quad & x< 0.5,~y< 0.5,\\[2pt]
(1.552, 0.0, -0.7169), \quad & x>0.5, y<0.5.
\end{array}
\right.
\end{array}
\end{equation}

\begin{figure}[htbp]
\centering
\includegraphics[width=0.7\textwidth]{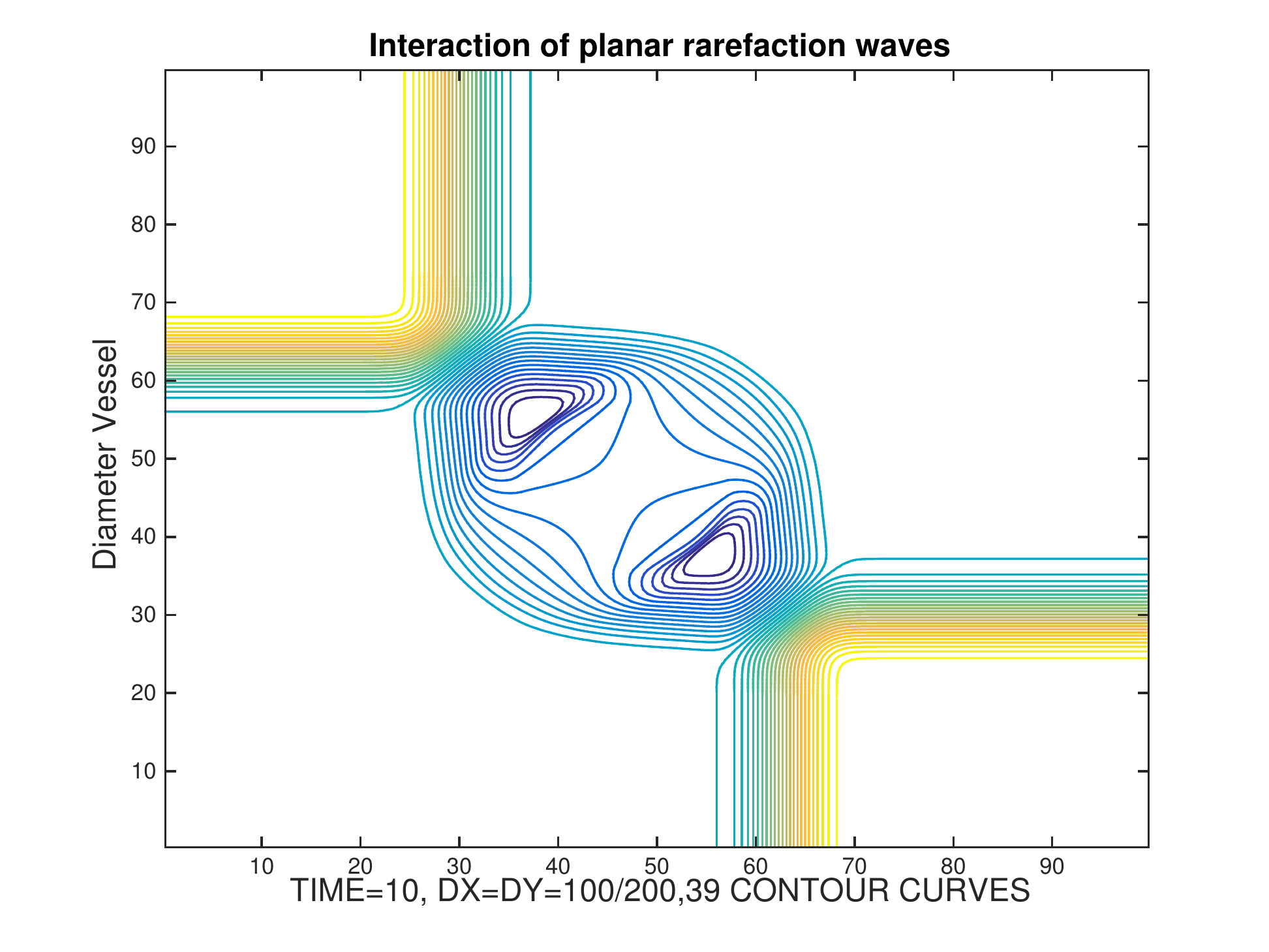}
\caption*{Fig. 8.6. Numerical results for four planar rarefaction waves.}
\end{figure}
We observe that two symmetric compressive waves in the interaction domain. See Fig. 8.6. This is a very typical two-dimensional phenomenon which also shows up in the 2D Euler cases.

\noindent
{\bf Example 7} (Interaction of shock waves).  This example shows the interaction of shocks.  The Riemann initial data is chosen as:
\begin{equation}\label{8.12}
\begin{array}{llll}
(A(x,y,0), u(x,y,0),v(x,y,0))=\left\{ \begin{array}{cc}
(1.0, 0.0, 0.0),\quad &  x> 0.5,~y> 0.5,\\[2pt]
(1.5, 0.6655, 0.0),\quad &  x< 0.5,~y> 0.5,\\[2pt]
(1.0, 0.6655, 0.6655),\quad & x< 0.5,~y< 0.5,\\[2pt]
(1.5, 0.0, 0.6655), \quad & x>0.5, y<0.5.
\end{array}
\right.
\end{array}
\end{equation}

\begin{figure}[htbp]
\centering
\includegraphics[width=0.7\textwidth]{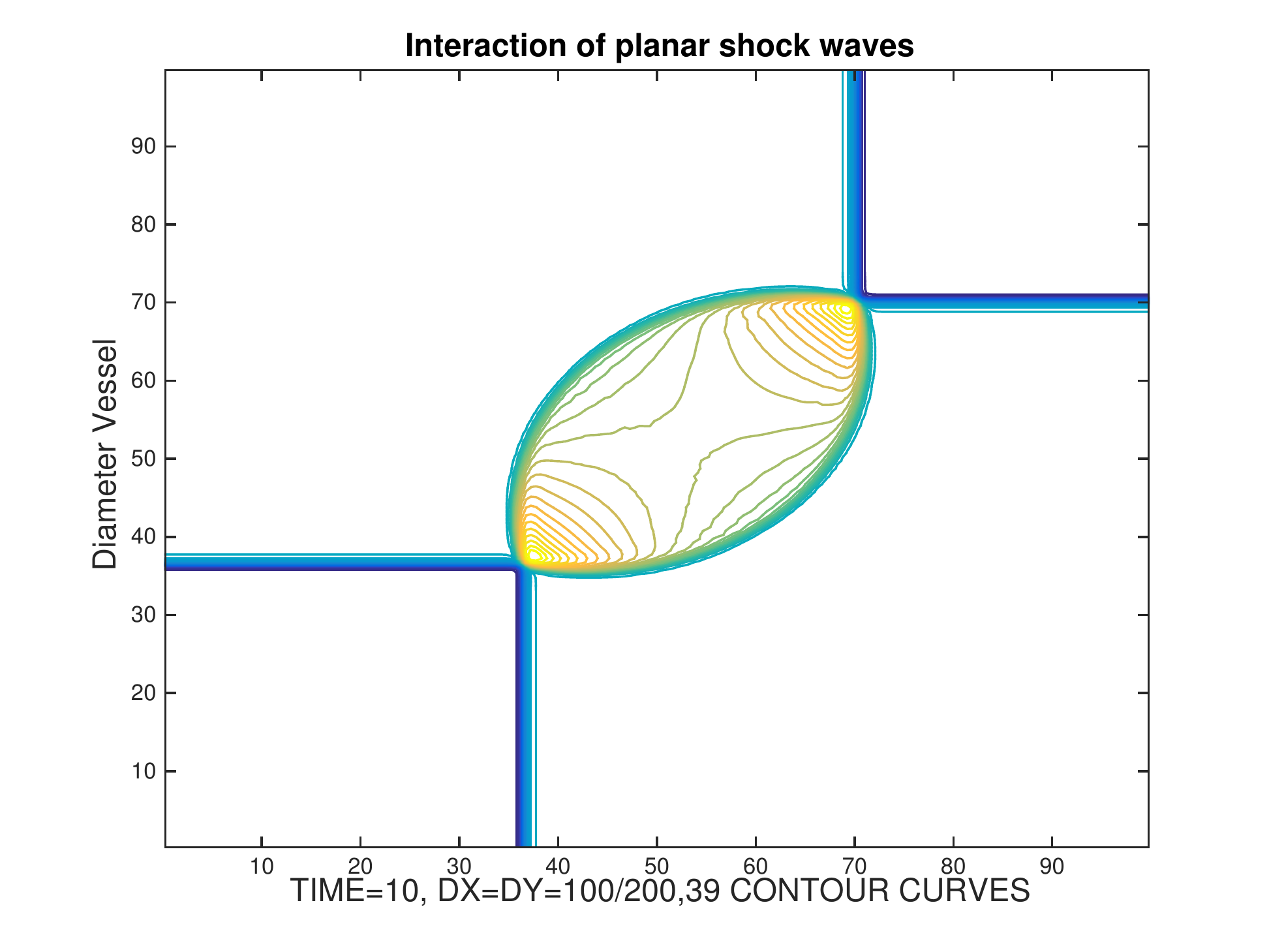}
\caption*{Fig. 8.7. Numerical results for four planar shock waves.}
\end{figure}
The four shock wave interact as time evolves, and a complicated wave pattern including shock reflection emerges. The numerical result is displayed in Fig. 8.7.

\noindent
{\bf Example 8} (Interaction of shock waves).  This example shows another shock interaction result. The Riemann initial data is chosen as:
\begin{equation}\label{8.13}
\begin{array}{llll}
(A(x,y,0), u(x,y,0),v(x,y,0))=\left\{ \begin{array}{cc}
(3.5, 0.0, 0.0),\quad &  x> 0.5,~y> 0.5,\\[2pt]
(1.428, 1.7849, 0.0),\quad &  x< 0.5,~y> 0.5,\\[2pt]
(0.46599, 1.7849, 1.7849),\quad & x< 0.5,~y< 0.5,\\[2pt]
(1.428, 0.0, 1.7849), \quad & x>0.5, y<0.5.
\end{array}
\right.
\end{array}
\end{equation}

\begin{figure}[htbp]
\centering
\includegraphics[width=0.7\textwidth]{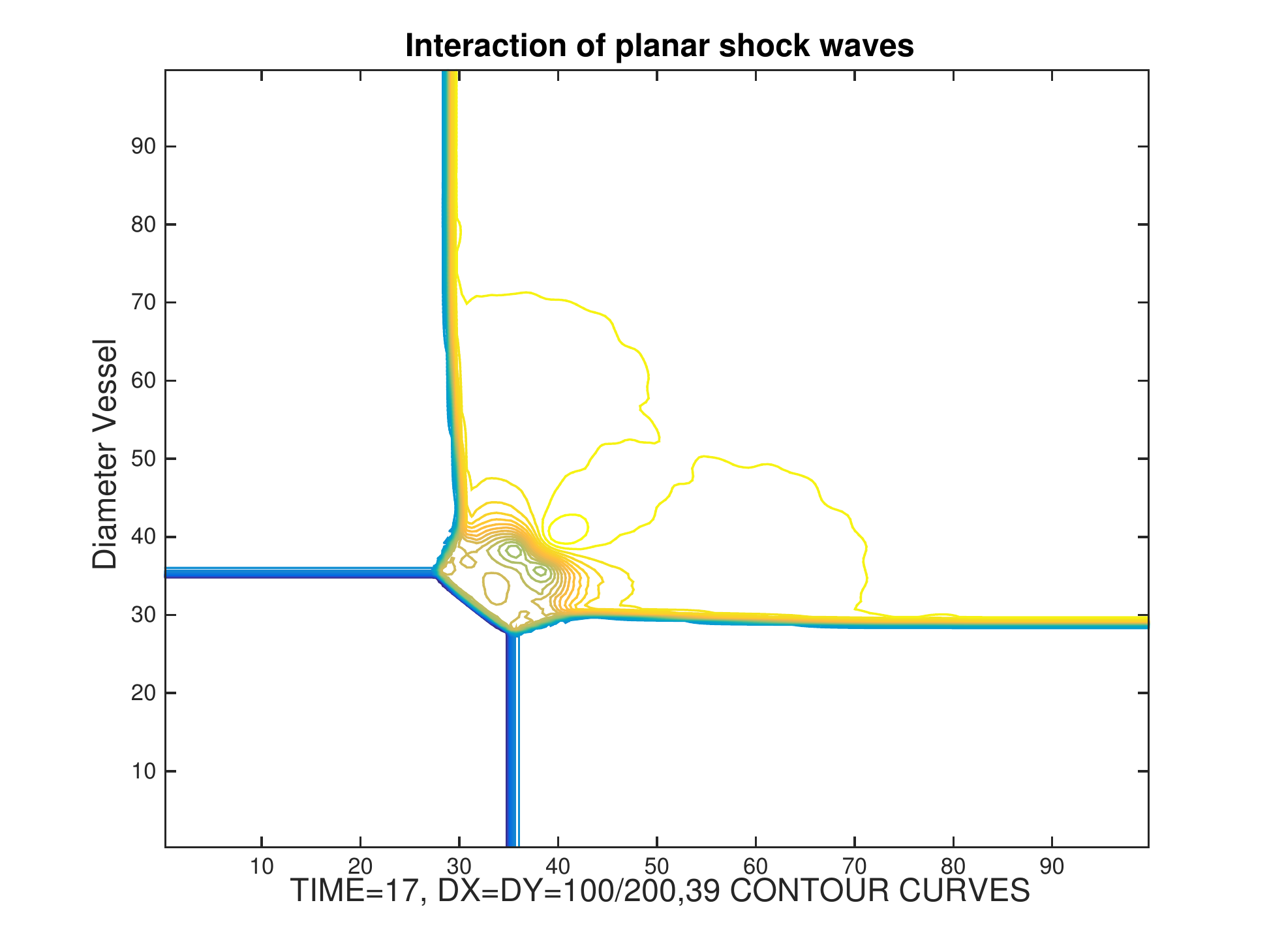}
\caption*{Fig. 8.8. Numerical results for four planar shock waves.}
\end{figure}
The four shock waves interact and the result includes triple points, Mach stems, etc. The numerical result is displayed in Fig. 8.8.

In summary, we apply the GRP scheme in the blood flow model in arteries. By using the Riemann invariants, we write such system in a quasi-diagonal form, which is used to resolve rarefaction wave. For shock wave, we use the Rankine-Hugoniot conditions to resolve the local GRP formulation. The acoustic approximation and sonic case are also given. We further extend the GRP scheme to the two dimensional case directly. The expected order is validated and numerical results are presented to justify the accuracy of the derived GRP scheme for the blood flow model.

\appendix

\section{Useful coefficients for GRP scheme}

\emph{A.1. This appendix gives the coefficients of the other two cases in Proposition 4.2, i.e., backward shock wave and forward rarefaction wave.}

{\bf (1). The coefficients of forward rarefaction wave.}\
For forward rarefaction wave, we have $\beta=u(0,\beta)+c(0,\beta)$, $\theta=c(0,\beta)/c_R$ in Proposition 4.2. The coefficients  are given by:
\begin{equation}\label{A1}
\left\{\begin{array}{ll}
\displaystyle
\left(a_R(0,\beta),b_R(0,\beta)\right)=\left(1, -\frac{c(0,\beta)}{A(0,\beta)}\right),\\[12pt]
\displaystyle  d_R(\beta)=-\frac{\beta-2c(0,\beta)}{2c(0,\beta)}\cdot (\beta-\phi_R)^{\frac{m+2}{2m}}\cdot \frac{\partial \phi}{\partial \alpha}(0,\beta)+\frac{\beta}{2c(0,\beta)}\cdot \left(\frac{k^{'}(0)}{\rho}-\frac{\left(\beta-c(0,\beta)\right)\cdot c(0,\beta)}{m}\cdot\frac{k^{'}(0)}{k(0)}\right),
\end{array}\right.
\end{equation}
where

\begin{equation}\label{A2}
\begin{array}{ll}
\displaystyle \frac{\partial \phi}{\partial \alpha}(0,\beta)=\frac{\partial \phi}{\partial \alpha}(0,\beta_R)+\frac{k'(0)}{\rho}\left(\frac{m+2}{m}c_R\right)^{-\frac{m+2}{2m}}(\theta^{-\frac{m+2}{2m}}-1)
\\[9pt]
~~\displaystyle  +\frac{\phi_R}{m-2}\frac{k'(0)}{k(0)}\left(\frac{m+2}{m}c_R\right)^{\frac{m-2}{2m}}(\theta^{\frac{m-2}{2m}}-1)+\frac{2}{(m+2)(3m-2)}\frac{k'(0)}{k(0)}\left(\frac{m+2}{m}c_R\right)^{\frac{3m-2}{2m}}(\theta^{\frac{3m-2}{2m}}-1)
\end{array}
\end{equation}

with
\begin{equation}\label{A3}
\frac{\partial \phi}{\partial \alpha}(0,\beta_R)=\left(\frac{m+2}{m}c_R\right)^{-\frac{m+2}{2m}}\cdot \left(\frac{k^{'}(0)}{\rho}-\frac{u_Rc_R}{m}\cdot \frac{k^{'}(0)}{k(0)}+2c_R\phi_R^{'}\right).
\end{equation}

{\bf (2). The coefficients of backward shock wave.}\
For backward shock wave, The coefficients of backward shock wave  in Proposition 4.2 are given by
\begin{equation}\label{A4}
\left\{\begin{array}{ll}
\displaystyle a_L=1-\frac{\sigma_L }{u_{*}^2-c_{*}^2}\cdot \left(u_{*}-A_*\Phi_3\right), \quad b_L=\frac{\sigma_L}{u_{*}^2-c_{*}^2}\cdot \left(\frac{c_{*}^2}{A_*}-u_{*}\Phi_3\right)-\Phi_3,\\[10pt]
\displaystyle d_L=\left(\sigma_L-u_L\right)u_L^{'}-\frac{k^{'}(0)}{\rho}\cdot \left(\left(\frac{A_L}{A_e}\right)^{m}-1\right)-\frac{mk(0)A_L^{m-1}A_L^{'}}{\rho A_e^{m}}+\Phi_4 \cdot \left((\sigma_L-u_L)A_L^{'}+A_Lu_L^{'}\right)\\[12pt]
\qquad\qquad\displaystyle +\frac{\sigma_L k'(0)}{\rho(u_*^{2}-c_*^{2})}\cdot \left(\left(\frac{A_*}{A_e}\right)^{m}-1\right)\cdot \left(u_*-A_* \Phi_3\right)-\sigma_Lk'(0)\Phi_k.
\end{array}\right.
\end{equation}
Here $\Phi_3$ and $\Phi_4$ are
\begin{equation}\label{A5}
\displaystyle  \Phi_3=\frac{\partial \Phi}{\partial A}\left(A_*;A_L,u_L\right),\quad \Phi_4=\frac{\partial \Phi}{\partial \overline{A}}\left(A_*;A_L,u_L\right),
\end{equation}
where
\begin{equation}\label{A6}
\left\{\begin{array}{ll}
\displaystyle
\frac{\partial \Phi}{\partial A}\left(A; \overline{A},\overline{u}\right)=\sqrt{\frac{mkA\overline{A}}{\rho(m+1)A_e^{m}(A-\overline{A})\left(A^{m+1}-\overline{A}^{m+1}\right)}}\frac{mA^{m+1}\overline{A}(A-\overline{A})+\overline{A}\left(A^{m+2}-\overline{A}^{m+2}\right)}{2\left(A\overline{A}\right)^{2}},\\[12pt]
\displaystyle \frac{\partial \Phi}{\partial \overline{A}}\left(A; \overline{A},\overline{u}\right)=\sqrt{\frac{mkA\overline{A}}{\rho(m+1)A_e^{m}(A-\overline{A})\left(A^{m+1}-\overline{A}^{m+1}\right)}}\frac{m\overline{A}^{m+1}A(\overline{A}-A)+A\left(\overline{A}^{m+2}-A^{m+2}\right)}{2\left(A\overline{A}\right)^{2}}.
\end{array}\right.
\end{equation}

\noindent
\emph{A.2. Sonic case.}
When the $t-$axis is located in the rarefaction wave associated with $u+c$. Then $\left(\partial u/\partial t\right)_*$ and  $\left(\partial A/\partial t\right)_*$ in Proposition 4.3 is given by
\begin{equation}\label{A.7}
\left\{\begin{array}{ll}
\displaystyle  \left(\frac{\partial u}{\partial t}\right)_{*}=\frac12\left(d_R(0)+k'(0)\cdot\left(\frac1{\rho}+\frac{u_{*}c_{*}}{mk(0)}\right)\right),\\[10pt]
\displaystyle  \left(\frac{\partial A}{\partial t}\right)_{*}=\frac{A_{*}}{2c_{*}}\left(-d_R(0)+k'(0)\cdot\left(\frac1{\rho}+\frac{u_{*}c_{*}}{mk(0)}\right)\right),
\end{array}\right.
\end{equation}
where $d_R(0)$ is given in (A.1).

\

{\bf Acknowledgements.} We would like to thank Professor Jiequan Li for sharing his code.

\bibliographystyle{amsplain}

 \end{document}